\documentclass{amsart}

\usepackage{amsfonts}
\usepackage{amsmath}
\usepackage{amsrefs}
\usepackage{amssymb}
\usepackage{amsthm}
\usepackage[english]{babel}
\usepackage{color}
\usepackage{epsfig}
\usepackage{enumerate}
\usepackage{enumitem}
\usepackage{graphicx}
\usepackage{hyperref}
\usepackage[utf8]{inputenc}
\usepackage{mathrsfs}
\usepackage{multicol}
\usepackage{nicefrac}
\usepackage{psfrag}
\usepackage{tikz-cd}
\usepackage{ushort}

\hypersetup{
    colorlinks,
    citecolor=blue, 
    filecolor=blue,
    linkcolor=blue,
    urlcolor=blue 
}

\newtheoremstyle{paragraph}{}{}{\normalfont}{}{}{}{\parindent}{}
\theoremstyle{plain}
\newtheorem{thm}{Theorem}[subsection]
\newtheorem{cor}[thm]{Corollary}
\newtheorem{lem}[thm]{Lemma}
\newtheorem{prop}[thm]{Proposition}

\theoremstyle{definition}
\newtheorem{df}[thm]{Definition}
\theoremstyle{remark}

\newtheorem{con}[thm]{Construction}
\newtheorem{rmk}[thm]{Remark}
\newtheorem{ex}[thm]{Example}
\theoremstyle{paragraph}
\newtheorem*{pp}{}

\newcommand{\R}{\mathbb R}

\newcommand{\C}{\mathbb C}
\newcommand{\Z}{\mathbb Z}
\newcommand{\N}{\mathbb N}
\newcommand{\s}{\mathbb S}

\renewcommand{\c}{\mathscr C}
\renewcommand{\a}{\mathscr A}
\renewcommand{\l}{\mathscr L}

\DeclareMathOperator{\id}{\it id}

\DeclareMathOperator{\mdim}{mdim}
\DeclareMathOperator{\ord}{ord}
\DeclareMathOperator{\uno}{\textup{\footnotesize 1}}
\DeclareMathOperator{\twedge}{\textstyle\bigwedge}
\DeclareMathOperator*{\dwedge}{\textstyle\bigwedge}
\DeclareMathOperator{\tprod}{\displaystyle\bigtimes}

\DeclareMathOperator{\toplus}{\textstyle\bigoplus}

\DeclareMathOperator{\tcup}{\textstyle\bigcup}

\DeclareMathOperator{\tcap}{\textstyle\bigcap}

\makeatletter
\providecommand{\bigsqcap}{%
  \mathop{%
    \mathpalette\@updown\bigsqcup
  }%
}
\newcommand*{\@updown}[2]{%
  \rotatebox[origin=c]{180}{$\m@th#1#2$}%
}
\makeatother

\makeatletter
\DeclareFontFamily{U}  {MnSymbolF}{}
\DeclareSymbolFont{symbolsMN}{U}{MnSymbolF}{m}{n}
\SetSymbolFont{symbolsMN}{bold}{U}{MnSymbolF}{b}{n}
\DeclareFontShape{U}{MnSymbolF}{m}{n}{
    <-6>  MnSymbolF5
   <6-7>  MnSymbolF6
   <7-8>  MnSymbolF7
   <8-9>  MnSymbolF8
   <9-10> MnSymbolF9
  <10-12> MnSymbolF10
  <12->   MnSymbolF12}{}
\DeclareFontShape{U}{MnSymbolF}{b}{n}{
    <-6>  MnSymbolF-Bold5
   <6-7>  MnSymbolF-Bold6
   <7-8>  MnSymbolF-Bold7
   <8-9>  MnSymbolF-Bold8
   <9-10> MnSymbolF-Bold9
  <10-12> MnSymbolF-Bold10
  <12->   MnSymbolF-Bold12}{}
\DeclareMathSymbol{\tbigtimes}{\mathop}{symbolsMN}{2}
\newcommand*{\bigtimes}{%
  \DOTSB
  \tbigtimes
  \slimits@ 
}
\makeatother

\makeatletter
\newsavebox\myboxA
\newsavebox\myboxB
\newlength\mylenA

\newcommand*\xbar[2][0.75]{%
    \sbox{\myboxA}{$\m@th#2$}%
    \setbox\myboxB\null
    \ht\myboxB=\ht\myboxA%
    \dp\myboxB=\dp\myboxA%
    \wd\myboxB=#1\wd\myboxA
    \sbox\myboxB{$\m@th\overline{\copy\myboxB}$}
    \setlength\mylenA{\the\wd\myboxA}
    \addtolength\mylenA{-\the\wd\myboxB}%
    \ifdim\wd\myboxB<\wd\myboxA%
       \rlap{\hskip 0.5\mylenA\usebox\myboxB}{\usebox\myboxA}%
    \else
        \hskip -0.5\mylenA\rlap{\usebox\myboxA}{\hskip 0.5\mylenA\usebox\myboxB}%
    \fi}
\makeatother

\newcommand{\ssim}{\approx}

\newcommand{\axci}{\textup{\ref{axc1}}}
\newcommand{\axcii}{\textup{\ref{axc2}}}
\newcommand{\axhi}{\textup{\ref{axh1}}}
\newcommand{\axhii}{\textup{\ref{axh2}}}
\newcommand{\axli}{\textup{\ref{axl1}}}
\newcommand{\axlii}{\textup{\ref{axl2}}}
\newcommand{\axui}{\textup{\ref{axu1}}}
\newcommand{\axuii}{\textup{\ref{axu2}}}
\newcommand{\axmi}{\textup{\ref{axm1}}}
\newcommand{\axmii}{\textup{\ref{axm2}}}
\newcommand{\axliii}{\textup{\ref{axl3}}}
\newcommand{\axuiii}{\textup{\ref{axu3}}}

\title{Abstract entropy and expansiveness}
\author{M. Achigar}
\date{\today}

\begin{document}

\begin{abstract}
We present a general definition of entropy in the setting of pre-ordered semigroups, extending the notion of topological entropy. From our definition, we obtain the basic properties exhibited by various entropy-like theories encountered in the literature, many of them being particular cases of our scheme. We show how to derive from the properties of the abstract entropy corresponding properties in the concrete cases. We also introduce a notion of expansiveness in this general context, extending the concept of expansive dynamical system, which is related to the entropy in a similar fashion as in the case of topological dynamics. Finally, as an application, we suggest some definitions of new entropies and expansiveness in concrete cases.
\end{abstract}

\maketitle

\tableofcontents

\section{Introduction}

In his 1965 seminal paper \cite{AKM65}, Adler, Konheim and McAdrew introduced the notion of \emph{topological entropy}, $h(T)$, for a continuous map $T\colon X\to X$ on a compact topological space $X$. Some of the key ingredients in the definition of this invariant are the open covers of $X$ (denoted here as $\alpha, \beta,\ldots$), preimages $T^{-1}(\alpha)$ under $T$ of open covers, and the meet $\alpha\wedge\beta$ of open covers (called \emph{join} and denoted $\vee$ in \cite{AKM65}). In that article several useful properties were developed, such as the \emph{logarithmic law} $h(T^n)=nh(T)$, and many others. The proofs of these properties relies, among other things, on the following basic fact: 
\begin{equation}\label{ecu_intro1}
 T^{-1}(\alpha\wedge\beta)=T^{-1}(\alpha)\wedge T^{-1}(\beta). 
\end{equation}
That is, if we consider the set $\c$ of all open covers, and the map $\lambda\colon\c\to\c$ given by $\lambda(\alpha)=T^{-1}(\alpha)$, then $\lambda$ preserves the operation $\wedge$. In fact, $\c$ is a semigroup with this operation so that $\lambda$ is a semigroup homomorphism, and the topological entropy can be defined in terms of $\lambda$. 

There are also other ingredients and basic facts supporting the definition and properties of $h(T)$. For example, the refinement relation $\alpha\prec\beta$ between open covers, the quantity $H(\alpha)$ of \cite{AKM65}*{Definition 1}, and the monotonicity of $\lambda$ and $H$, but we prefer to stress only property (\ref{ecu_intro1}) to keep this introduction simple.

Another theory very similar to that of topological entropy is the theory of \emph{mean dimension}. The mean dimension, $\mdim(T)$, of a continuous map $T$ as before, is introduced in \cite{LiW00}. To obtain this definition we simply need to replace the quantity $H(\alpha)$ in the definition of $h(T)$ by the quantity $D(\alpha)$ of \cite{LiW00}*{Definition 2.1}. As $H$ and $D$ share the relevant basic properties, both theories also share various properties (and proofs). This is also the case of the theory of \emph{algebraic entropy}, introduced in \cite{Weiss74}. To get the definition of the algebraic entropy of an endomorphism $\varphi$ of an abelian group $G$, we need to replace $\c$ by the set of all finite subgroups of $G$, $\wedge$ by the sum $+$ of subgroups, $\prec$ by the relation $\supseteq$, $\lambda$ by the map taking a finite subgroup to its image under $\varphi$, and the quantity $H$ by the logarithm of the cardinality of the subgroups. In this two examples and many others the map $\lambda\colon\c\to\c$ always preserves the semigroup opertaion of $\c$. However, this is not the case in some natural contexts.

For example, suppose that we want to define a forward version of the topological entropy, that is, to consider the map $\lambda\colon\c\to\c$ given by $\lambda(\alpha)=T(\alpha)$. In first place, it is necessary to assume that $T\colon X\to X$ is an open onto map, to guarantee that $\lambda$ carries open covers to open covers. With these changes one can define the \emph{forward topological entropy} of such a map $T$ following the same method as for the topological entropy. However, this time $\lambda$ does not necessarily preserves the operation $\wedge$, but instead we can ensure that
\begin{equation}\label{ecu_intro2}
T(\alpha\wedge\beta)\prec T(\alpha)\wedge T(\beta).
\end{equation}
\enlargethispage{5mm}
That is, $\lambda$ is not a semigroup homomorfism anymore, but it satisfies a weaker condition resembling subadditivity. We will call this type of maps \emph{lower maps}. Surprisingly, it is still possible to prove various of the topological entropy theory properties in this context. To show that is one of the main contributions of the present work.

In Section \ref{sec:entropy} we consider the situation above from an abstract viewpoint. Starting with a lower map $\lambda$ acting on a general semigroup $\c$, we define his entropy and then we derive the properties. Later, in Section \ref{sec:ejemplos}, we give few examples of known entropy-like theories as particular cases of our abstract entropy, and show how to obtain some of their properties from the abstract ones. 
In \cite{DGB19}, Dikranjan and Giordano Bruno made a similar delevopement assuming that the map $\lambda$ does preserves the operation of $\c$. In  \S\ref{subsec:fhtop} and \S\ref{subsec:algebraic} we also suggest some possibly useful new entropy theories which are particular cases of our abstract entropy.

The second main contribution of this article is the introduction of the concept of expansivity in the abstract setting of a general semigroup $\c$. The usual notion of expansivity of a dynamical system on a compact metric space, is a particular case of our abstract expansivity, and corresponds to the case when $\c$ is the set of all open covers of the space, as in the beginning of this introduction. 

Moreover, abstract expansivity and entropy interacts in the same way as in the theory of topological dynamical systems. A useful tool for computing topological entropy for an expansive dynamical system is the following fact: the entropy of the system coincides with the entropy of size $\varepsilon>0$, if this size is smaller than an expansivity constant. We will prove an analogue of this result at the abstract level of a general semigroup $\c$. These topics are covered in Section \ref{sec:expansiveness}.

Finally, in Section \ref{subsec:algebraic} we consider the special case of our abstract expansivity when the semigroup $\c$ is the corresponding to the theory of algebraic entropy, obtaining a possibly useful notion of expansivity for group endomorphisms, which of course interacts with the entropy in the correct way.

\subsection{Acknowledgments}

\begin{pp}
 The author would like to express gratitude to Professor Anna Giordano Bruno for her valuable and generous comments on algebraic entropy and expansiveness. Additionally, the author appreciates her kindness and willingness to exchange ideas about this research.
 
 This research was partially funded by the Sistema Nacional de Investigadores - Agencia Nacional de Investigación e Innovación, Uruguay.
\end{pp}

\subsection{Notation and conventions}\label{subsec:notation}

\begin{pp}
 Let $X$ be a set, $\alpha$ and $\beta$ families of subsets of $X$, and $A\subseteq X$. We say that $\alpha$ is a \emph{refinement} of $\beta$, denoted $\alpha\prec\beta$, iff for every $A\in\alpha$ there exists $B\in\beta$ such that $A\subseteq B$. We also write $A\prec\beta$ to mean that $\{A\}\prec\beta$. The \emph{meet} of $\alpha$ and $\beta$ is $\alpha\wedge\beta=\{A\cap B:A\in\alpha,\,B\in\beta\}$. We also write $A\wedge\beta=\{A\}\wedge\beta$. The union of the members of $\alpha$ is denoted $\bigcup\alpha$. We say that $\alpha$ is \emph{cover} iff $\bigcup\alpha=X$. If $A\subseteq\bigcup\alpha$ then $\alpha$ is called a \emph{cover of $A$}.

 Let $f\colon X\to Y$ be a map. If $x\in X$ the image of $x$ under $f$ is denoted $fx$ or $f(x)$. If $A\subseteq X$ and $\alpha$ is a family of subsets of $X$ we write $fA=f(A)=\{fx:x\in A\}$ and $f\alpha=f(\alpha)=\{fA:A\in\alpha\}$. For $B\subseteq Y$ and a family $\beta$ of subsets of $Y$ we denote $f^{-1}B=f^{-1}(B)=\{x\in X:fx\in B\}$ and $f^{-1}\beta=f^{-1}(\beta)=\{f^{-1}B:B\in\beta\}$. For a point $y\in Y$ we sometimes write $f^{-1}y=f^{-1}\{y\}$. If $X=Y$ and $n\in\N=\{0,1,2,\ldots\}$ we denote as usual $f^n=f\circ\cdots\circ f$ ($n$ times) if $n\geq1$, $f^0=\id_X$ is the identity function on $X$, and $f^{-n}=(f^{-1})^n$ if $f$ is bijective. 

 We denote $\R^+=[0,\infty]$ the set of non-negative real numbers together with the symbol $\infty$. We extend the order and the operations of real numbers agreeing that $a\leq\infty$ and $a+\infty=\infty+a=a\cdot\infty=\infty\cdot a=\infty$ for all $a\in\R^+$. Note that in particular $0\cdot\infty=\infty$. Functions like ``$\sup$'' or ``$\limsup$'' naturally extends taking values in $\R^+$.
We also define $\N^+=\N\cup\{\infty\}\subseteq\R^+$. If $A$ is a set then $|A|\in\N^+$ stands for the cardinality of $A$ if it is finite or $\infty$ otherwise. We denote $\log\colon\N^+\to\R^+$ the natural logarithm function extended as $\log 0=0$ and $\log\infty=\infty$.
\end{pp}

\section{Abstract entropy}\label{sec:entropy}

In this section we introduce an abstract version of entropy, inspired in the definition of topological entropy of a continuous map given in \cite{AKM65}, and we prove several of its properties. The entropy theory we develop here is similar to the \emph{semigroup entropy} of \cite{DGB19}*{\S2 and \S3}, although there are some important differences. In \cite{DGB19} the semigroup entropy is studied mainly for (contractive) endomorphisms of (normed) semigroups, while here we consider (normed) semigroups with a preorder which plays a central role and allow us to study entropy for a more general class of (contractive) maps called \emph{lower maps}.

\subsection{Entropy spaces}\label{subsec:coventspaces}

\begin{pp}
 The first thing to abstract from the definition of topological entropy of \cite{AKM65} is the role played by the collection $\c_X$ of all open covers of a (compact) topological space $X$. In $\c_X$ we have a pre-order $\prec$ and an operation $\wedge$ given by the refinement relation and the meet of open covers,\footnote{Note that, according to \S\ref{subsec:notation}, our $\prec$ has the opposite meaning than in \cite{AKM65}*{Definition 3} and that our $\wedge$ is denoted $\vee$ and called \emph{join} in \cite{AKM65}*{Definition 2}.} and both verifies a series of basic properties such as \cite{AKM65}*{Property 00, 0 and 1}. We also have a non-negative real value $H(\alpha)$ associated to each cover $\alpha\in\c_X$, called the \emph{entropy} of $\alpha$, in \cite{AKM65}*{Definition 1}, that relates to $\prec$ and $\wedge$ through properties like \cite{AKM65}*{Property 2 and 4}. Inspired on this we introduce the following definitions.
\end{pp}

\begin{df}\label{def:cov_space}
 A \emph{cover space} is a 3-tuple $(\c,\prec,{}\wedge)$ where $\c$ is a non-empty set, $\prec$ a pre-order in $\c$ (a transitive and reflexive relation) and $\wedge$ an associative binary operation on $\c$ satisfying the following conditions:
 \begin{enumerate}[label={\textsc{[c\footnotesize\arabic*}]}]
  \item\label{axc1} $\alpha\wedge\beta\prec\alpha$ and $\alpha\wedge\beta\prec\beta$ for all $\alpha,\beta\in\c$.
  \item\label{axc2} $\alpha\wedge\beta\prec\alpha'\wedge\beta'$ if $\alpha\prec\alpha'$ and $\beta\prec\beta'$, where $\alpha,\alpha',\beta,\beta'\in\c$.
 \end{enumerate}
 The structure $(\c,\prec,{}\wedge)$, also called \emph{pre-ordered semigroup}, will be denoted simply as $\c$. In this context an element of $\c$ is called a \emph{cover}, $\prec$ is the \emph{refinement relation} and $\wedge$ the \emph{meet operation}. Given $\alpha,\beta\in\c$ such that $\alpha\prec\beta$ (also written $\beta\succ\alpha$) we say that $\alpha$ \emph{is finer than}/\emph{is a refinement of}/\emph{refines} $\beta$.
\end{df}

\begin{rmk}\label{obs:cunatoria} 
 If $\c$ is a cover space and $\alpha_1,\ldots,\alpha_n\in\c$, conditions \axci{} and \axcii{} implies that adding factors to the product $\twedge_{k=1}^n\alpha_i=\alpha_1\wedge\cdots\wedge\alpha_n$ at any place will result in a finer cover. For example, $\alpha_1\wedge\beta\wedge\alpha_2\prec\alpha_1\wedge\alpha_2$ if $\alpha_1,\alpha_2,\beta\in\c$. 
\end{rmk}

\begin{df}\label{def:equiv_conmut_idemp_meet}
 Given a cover space $\c$ consider the equivalence relation $\sim$ on $\c$ defined for $\alpha,\beta\in\c$ as $\alpha\sim\beta$ iff $\alpha\prec\beta$ and $\alpha\succ\beta$. We say that $\c$ (or $\wedge$) is \emph{commutative} iff $\alpha\wedge\beta\sim\beta\wedge\alpha$ for all $\alpha,\beta\in\c$. We call $\c$ a \emph{meet cover space}, also called \emph{pre-lower semilattice}, iff for all $\alpha,\beta\in\c$ the cover $\alpha\wedge\beta$ is a \emph{greatest lower bound} for $\{\alpha,\beta\}$, that is, if $\gamma\prec\alpha$ and $\gamma\prec\beta$ for some $\gamma\in\c$ then $\gamma\prec\alpha\wedge\beta$.
\end{df}

\begin{rmk}\label{obs:icomut_idemp_meet}
 For a cover space $\c$ it is not difficult to show that to be a meet cover space is equivalent to the condition: $\alpha\wedge\alpha\sim\alpha$ for all $\alpha\in\c$, which in turn is equivalent to the property: if $\alpha,\beta\in\c$ and $\alpha\prec\beta$ then $\alpha\sim\alpha\wedge\beta$. It is also clear that a meet cover space is automatically commutative. Hence, in a product of covers of a meet cover space we can cancel repeated factors or in general factors refined by other factors of the product, for example $\alpha\wedge\beta\wedge\alpha\wedge\gamma\sim\alpha\wedge\beta$ if $\beta\prec\gamma$.
\end{rmk}

\begin{rmk}\label{obs:pre-ordered_monoid}
 If a cover space $\c$ has a \emph{unit} ${\uno}\in\c$, that is, $\alpha\wedge{\uno}\sim{\uno}\wedge\alpha\sim\alpha$ for all $\alpha\in\c$, then by condition \axcii{} we can replace condition \axci{} in Definition \ref{def:cov_space} by the simpler one: $\alpha\prec{\uno}$ for all $\alpha\in\c$. This structure will be called \emph{unital cover space} or \emph{pre-ordered monoid}. If a cover space $\c$ has no unit element (or even if it has) we can add one, say ${\uno}\not\in\c$, to it and get a unital cover space $\c\cup\{{\uno}\}$ extending the relation $\prec$ and the binary operation $\wedge$ in the obvious way. 
\end{rmk}

\begin{rmk}\label{obs:partially ordered_monoid}
 Given a cover space $\c$ let $\c_1=\c/{\sim}$ be the quotient space by the equivalence relation $\sim$ of Definition \ref{def:equiv_conmut_idemp_meet}, and let $[\alpha]\in\c_1$ denote the equivalence class of $\alpha\in\c$. Define the relation $\prec_1$ on $\c_1$ by $[\alpha]\prec_1[\beta]$ iff $\alpha\prec\beta$, for $\alpha,\beta\in\c$, and consider the binary operation $\wedge_1$ on $\c_1$ given by $[\alpha]\wedge_1[\beta]=[\alpha\wedge\beta]$ if $\alpha,\beta\in\c$. As can be easily checked $\prec_1$ is a partial order (a transitive, reflexive and antisymmetric relation) and $(\c_1,\prec_1,{}\wedge_1)$ is a cover space. Then, without loss of generality, we could have assumed that $\prec$ is a partial order, rather than a pre-order, in the definition of cover spaces. But as in many examples what naturally arise are pre-orders we prefer the given definition.
\end{rmk}

\begin{ex}\label{ej:latticecoverspace}
 Let $\l$ be a \emph{bounded distributive lattice} as in \cite{Ach21}*{Definition 2.1} and $\c=\c(\l)$ the set of \emph{covers} in $\l$ endowed with the pre-order $\prec$ and the binary operation $\wedge$ as in \cite{Ach21}*{Definition 2.2}. Then $\c$ is clearly a meet cover space. 
\end{ex}

\begin{df}\label{def:ent_space}
 An \emph{entropy space} is a 4-tuple $(\c,\prec,{}\wedge,h)$ where $(\c,\prec,{}\wedge)$ is a cover space and $h\colon\c\to\R^+$ a function such that the following conditions hold:
 \enlargethispage{5mm}
 \begin{enumerate}[label={\textsc{[h\footnotesize\arabic*}]}]
  \item\label{axh1} 
  $h(\alpha)\geq h(\beta)$ if $\alpha\prec\beta$, where $\alpha,\beta\in\c$.
  \item\label{axh2} 
  $h(\alpha\wedge\beta)\leq h(\alpha)+h(\beta)$ for all $\alpha,\beta\in\c$.
 \end{enumerate}
 The function $h$ is called \emph{entropy function}. An entropy space will be denoted simply as $(\c,h)$ or even $\c$. We say that $\c$ (or $h$) is \emph{finite} iff $h(\alpha)<\infty$ for all $\alpha\in\c$.
\end{df}

The concept of entropy function, which is inspired by the quantity $H$ of \cite{AKM65}*{Definition 1}, corresponds to the \emph{subadditive norms} of \cite{DGB19}*{Definition 2.4(a)}. Therefore, the entropy spaces and the meet entropy spaces are similar to the \emph{normed pre-ordered semigroups} and \emph{normed pre-semilattices}, respectively, of \cite{DGB19}*{Definition 2.12(a)}, except that condition \axci{} is not imposed (see \cite{DGB19}*{Definition 2.7}).

\begin{rmk} 
 Note that from conditions \axci{}, \axhi{} and \axhii{} on an entropy space $\c$ we have $\max\{h(\alpha),h(\beta)\}\leq h(\alpha\wedge\beta)\leq h(\alpha)+h(\beta)$, for all $\alpha,\beta\in\c$. Also, by condition \axhi{}, if $\alpha,\beta\in\c$ and $\alpha\sim\beta$ then $h(\alpha)=h(\beta)$.
\end{rmk}

\begin{rmk}\label{obs:h_unital_po}
 If $\c$ is a unital entropy space with unit ${\uno}\in\c$ we could have $h({\uno})\neq0$. Redefining $h({\uno})=0$ we also get an entropy space. \emph{Then we will always assume that $h({\uno})=0$ in a unital entropy space}. If an entropy space has no unit we can add one to it as in Remark \ref{obs:pre-ordered_monoid} and define $h({\uno})=0$ getting a unital entropy space. On the other hand, the construction of Remark \ref{obs:partially ordered_monoid} can be done in an entropy space defining $h_1([\alpha])=h(\alpha)$ if $\alpha\in\c$ to get an entropy space $\c_1$.
\end{rmk}

\subsection{Morphisms}\label{subsec:maps}

\begin{pp}
 If $T\colon X\to X$ is a continuous function on a (compact) topological space $X$, then the definition of the topological entropy of $T$ given in \cite{AKM65}*{p.\,310} is formulated in terms of taking preimages under $T$ of open covers of $X$. That is, in terms of the map $\lambda_T\colon\c_X\to\c_X$ defined as $\lambda_T\alpha=T^{-1}\alpha$ if $\alpha\in\c_X$, where $\c_X$ is the collection of all open covers of $X$. This map verifies some properties such as \cite{AKM65}*{Property 5, 6 and 7}, in particular $\lambda_T$ preserves the meet operation on open covers. In the next definition we consider properties of this type at the abstract level of entropy spaces introducing various classes of maps, some of them allowing non $\wedge$-preserving maps.
\end{pp}

\begin{df}\label{def:lmapumap}
 Let $\c_1,\c_2$ be entropy spaces and $\lambda\colon\c_1\to\c_2$ a \emph{monotone} map, that is, $\lambda\alpha\prec\lambda\beta$ if $\alpha,\beta\in\c_1$ and $\alpha\prec\beta$. Then $\lambda$ is called \emph{lower map} (or  \emph{l-map}) iff
 \begin{enumerate}[label={\textsc{[l\footnotesize\arabic*}]}]
  \item\label{axl1}
  $\lambda(\alpha\wedge\beta)\prec\lambda\alpha\wedge\lambda\beta$ for all $\alpha,\beta\in\c_1$, and
  \item\label{axl2}
  $h(\lambda\alpha)\leq h(\alpha)$ for all $\alpha\in\c_1$.
 \end{enumerate}
 We say that $\lambda$ is an \emph{upper map} (\emph{u-map} for short) iff
 \begin{enumerate}[label={\textsc{[u\footnotesize\arabic*}]}]
  \item\label{axu1}
  $\lambda(\alpha\wedge\beta)\succ\lambda\alpha\wedge\lambda\beta$ for all $\alpha,\beta\in\c_1$, and
  \item\label{axu2}
  $h(\lambda\alpha)\geq h(\alpha)$ for all $\alpha\in\c_1$.
 \end{enumerate}
 We call $\lambda$ a \emph{homomorphism} iff 
 \begin{enumerate}[label={\textsc{[m\footnotesize\arabic*}]}]
  \item\label{axm1}
  $\lambda(\alpha\wedge\beta)\sim\lambda\alpha\wedge\lambda\beta$ for all $\alpha,\beta\in\c_1$, and
  \item\label{axm2} 
  $h(\lambda\alpha)=h(\alpha)$ for all $\alpha\in\c_1$.
 \end{enumerate}
 If the monotone map $\lambda$ satisfies condition \axmi{} we call it a \emph{morphism} (only cover space structure involved). If it verifies conditions \axmi{} and \axlii{} then it is called \emph{lower morphism} (or \emph{l-morphism}), and it is called \emph{upper morphism} (or \emph{u-morphism}) iff $\lambda$ satisfies conditions \axmi{} and \axuii{}. On the other hand, $\lambda$ is said to be an \emph{isomorphism} iff $\lambda$ is a bijective homomorphism such that its inverse map $\lambda^{-1}$ is also a homomorphism. Finally, if $\c_1$ and $\c_2$ are unital cover spaces,  with units denoted ${\uno}$, and $\lambda\colon\c_1\to\c_2$ verifies $\lambda{\uno}={\uno}$ then $\lambda$ is called a \emph{unital} map.
\end{df}

We will mainly focus on studying entropy for lower maps, while the upper maps will serve to connect the lower maps and compare their entropies in \S\ref{subsec:comparison}. As we noted at the beginning of \S\ref{sec:entropy}, the semigroup entropy presented in \cite{DGB19} is applicable to a broader class of semigroups (without a preorder) than the ones we consider here. However, in exchange, our theory can be developed for a wider class of maps that are not necessarily $\wedge$-preserving. This generality is essential in the examples we give in \S\ref{subsec:fhtop} and \S\ref{subsec:algebraic}. The following diagram illustrates the relationships between the different classes of monotone maps introduced in Definition \ref{def:lmapumap}.

\begin{center}
\vspace{1ex}
\tikzset{commutative diagrams/row sep/normal=-2ex}
\begin{tikzcd}[cells={nodes={draw=black}}]
& & \begin{smallmatrix}
   \displaystyle\text{upper morphism}\\
   \displaystyle \axmi{} + \axuii{} \end{smallmatrix}
 \ar[r, Rightarrow, start anchor={east}, end anchor={west}, shorten=1mm] & \begin{smallmatrix}
   \displaystyle\text{upper\phantom{l}\hspace{.15mm}map}\\
   \displaystyle \axui{} + \axuii{} \end{smallmatrix} \\
\begin{smallmatrix}
   \displaystyle\text{morphism}\\
   \displaystyle \axmi{} \end{smallmatrix} & \begin{smallmatrix}\displaystyle\text{homomorphism}\\
   \displaystyle \axmi{} + \axmii{}\end{smallmatrix} \ar[ur, Rightarrow, start anchor={[yshift=-1.5ex]north east}, end anchor={west}, shorten=1mm] \ar[dr, Rightarrow, start anchor={[yshift=1.5ex]south east}, end anchor={west}, shorten=1mm] \ar[l, Rightarrow, start anchor={west}, end anchor={east}, shorten=1mm]\\
& & \begin{smallmatrix}\displaystyle\text{\hspace{.42mm}lower morphism\hspace{.42mm}}\\
   \displaystyle \axmi{} + \axlii{} \end{smallmatrix} \ar[r, Rightarrow, start anchor={east}, end anchor={west}, shorten=1mm] & \begin{smallmatrix}
   \displaystyle\text{\hspace{.42mm}lower map\hspace{.42mm}}\\
   \displaystyle \axli{} + \axlii{}
  \end{smallmatrix}
\end{tikzcd}
\vspace{1ex}
\end{center}

\begin{rmk} 
 Let $\lambda\colon\c_1\to\c_2$ be a lower map and $\alpha_1,\ldots,\alpha_n\in\c_1$. Then condition \axcii{} implies that condition \axli{} generalizes to $\lambda\twedge_{k=1}^n\alpha_i\prec\twedge_{k=1}^n\lambda\alpha_i$. Analogously, $\lambda\twedge_{k=1}^n\alpha_i\succ\twedge_{k=1}^n\lambda\alpha_i$ if $\lambda$ is a upper map, with equivalence in the case morphisms.
\end{rmk}

\begin{rmk}\label{obs:bi-nomo}
 Let $\lambda\colon\c_1\to\c_2$ be a \emph{bi-monotone} map between entropy spaces, that is, $\lambda$ is monotone, bijective and $\lambda^{-1}$ is monotone. Then, $\lambda$ is a lower map iff $\lambda^{-1}$ is an upper map. In particular, if $\lambda$ is a homomorphism then it is an isomorphism. 
\end{rmk}

\begin{rmk}\label{obs:categoria}
 Compositions of l-maps are l-maps because l-maps are monotone. Then, with the entropy spaces as objects and the l-maps as arrows we obtain a category. Similarly, if we take as arrows the l-morphisms, u-maps, u-morphisms or homomorphisms we also get categories. Note that in these five categories the category isomorphisms are precisely the isomorphisms introduced in Definition \ref{def:lmapumap}. Compositions of morphisms are morphisms as well, again by monotonicity.
\end{rmk}

\begin{rmk}\label{obs:potencias}
 By Remark \ref{obs:categoria} if $\lambda\colon\c\to\c$ is an l-map, l-morphism, u-map, u-morphism, homomorphism or a morphism, then $\lambda^n$ is of the same class for all $n\in\N$. If $\lambda$ is an isomorphism then $\lambda^n$ is an isomorphism for all $n\in\Z$. 
\end{rmk}

\subsection{Entropy}\label{subsec:entropy}

\begin{pp}
 In this section we define the abstract entropy and show some basic facts that will be used later. At the end we show a result not stressed in the literature as far as we know, that could be useful in further developments.
\end{pp}

\begin{df}\label{def:cunatoria}
 Let $\c$ be a cover space, $\lambda\colon\c\to\c$ a map, $\alpha\in\c$ and $m,n\in\N$ such that $n\leq m$. We define $\alpha_n^m[\lambda]=\alpha_n^m=\twedge_{k=n}^m\lambda^k\alpha$. If $\lambda$ is a bijective map the above definition makes sense for $m,n\in\Z$.
\end{df}

Recall that the order of the factors matters in Definition \ref{def:cunatoria} because a general cover space is not assumed to be commutative.

\begin{df}\label{def:entropy} 
 Let $\c$ be an entropy space, $\lambda\colon\c\to\c$ a map and $\alpha\in\c$. We define the \emph{entropy of $\lambda$ relative to the cover $\alpha$} and the \emph{entropy of $\lambda$} as
 \[
 h(\lambda,\alpha)=\textstyle\limsup_n\tfrac1nh(\alpha_0^{n-1})\quad\text{and}\quad h(\lambda)=\textstyle\sup_{\alpha\in\c}h(\lambda,\alpha),
 \]
 respectively, where the $\limsup_n$ and $\sup_\alpha$ functions can take values in $\R^+=[0,\infty]$.
\end{df}

\begin{rmk}\label{obs:h_infinito}
 In the context of Definition \ref{def:entropy} we have that the $\R^+$-valued sequence $\bigl(h(\alpha_0^{n-1})\bigr)$ is increasing. Indeed, if $n\geq1$, as $\alpha_0^n\prec\alpha_0^{n-1}$ by condition \axci{}, we get $h(\alpha_0^n)\geq h(\alpha_0^{n-1})$ by condition \axhi{}. Hence, for a non-finite entropy space we always have $h(\lambda)=\infty$, because $h(\lambda,\alpha)=\infty$ if $\alpha$ is a cover such that $h(\alpha)=\infty$.
\end{rmk}

\begin{lem}[\cite{AKM65}*{Property 10}]\label{lem:hlambda_monot}
 Let $\c$ be an entropy space, $\lambda\colon\c\to\c$ a monotone map, and $\alpha,\beta\in\c$ such that $\alpha\prec\beta$, then $h(\lambda,\alpha)\geq h(\lambda,\beta)$.
\end{lem}

\begin{proof}
 As $\alpha\prec\beta$ then $\lambda^k\alpha\prec\lambda^k\beta$ for all $k\in\N$ because $\lambda$ is monotone. Then, for all $n\geq1$, we have $\alpha_0^{n-1}\prec\beta_0^{n-1}$ by condition \axcii, and hence $h(\alpha_0^{n-1})\geq h(\beta_0^{n-1})$ by condition \axhi. Dividing by $n$ and taking $\limsup_n$ we get $h(\lambda,\alpha)\geq h(\lambda,\beta)$.
\end{proof}

As a consequence of Lemma \ref{lem:hlambda_monot}, to compute $h(\lambda)$ we will be interested in the asymptotic behaviour of $h(\lambda,\alpha)$ as  the covers $\alpha$ become finer. Note that in this direction (relation $\succ$) every cover space $\c$ is a \emph{directed set} by condition \axci{}, that is, for every $\alpha,\beta\in\c$ there exists $\gamma\in\c$ (namely $\gamma=\alpha\wedge\beta$) such that $\gamma\prec\alpha$ and $\gamma\prec\beta$. We will consider $\c$ as a directed set in this way when taking limits. Accordingly we will use the term \emph{cofinal} in the sense of the following definition.

\begin{df}\label{def:cofinal}
 Let $\c$ be a cover space. A subset $\c'\subseteq\c$ is called \emph{cofinal} iff for all $\alpha\in\c$ there exists $\beta\in\c'$ such that $\beta\prec\alpha$. A map $\lambda\colon\c_1\to\c_2$ between cover spaces is called \emph{cofinal map} iff its image $\lambda(\c_1)$ is a cofinal subset of $\c_2$.
\end{df}

The next easy consequence of Lemma \ref{lem:hlambda_monot} allows us to replace the ``$\sup_{\alpha\in\c}$'' of Definition \ref{def:entropy} by a ``$\lim_{\alpha\in\c'}$'' for a cofinal subset $\c'\subseteq\c$ in the monotone case.

\begin{lem}[\cite{AKM65}*{Property 12}]\label{lem:cofinal} 
 Let $\c$ be an entropy space, $\lambda\colon\c\to\c$ a monotone map and $\c'\subseteq\c$ cofinal. Then $h(\lambda)=\sup_{\alpha\in\c'}h(\lambda,\alpha)=\lim_{\alpha\in\c'}h(\lambda,\alpha)$.
\end{lem}

Next we generalize \cite{AKM65}*{Property 8} and the subadditive property in its proof.

\begin{lem}\label{lem:subatid} 
 Let $\c$ be an entropy space, $\lambda\colon\c\to\c$ a lower map and $\alpha\in\c$. Then the sequence $(a_n)$ given by $a_n=h(\alpha_0^{n-1})$ if $n\geq1$ is subadditive, that is, $a_{n+m}\leq a_n+a_m$ for all $n,m\geq1$. 
\end{lem}

\begin{proof}
 For $n,m\geq1$ let $s=n+m$ denote their sum. Then we  can estimate
 \begin{equation*}
 \begin{split}
  a_{n+m}
  &=h\bigl(\twedge_{k=0}^{s-1}\lambda^k\alpha\bigr)
  \leq h\bigl(\twedge_{k=0}^{n-1}\lambda^k\alpha\bigr)+h\bigl(\twedge_{i=n}^{s-1}\lambda^k\alpha\bigr)\\
  &\leq a_n+h\bigl(\lambda^n\twedge_{k=0}^{m-1}\lambda^k\alpha\bigr)
  \leq a_n+h\bigl(\twedge_{k=0}^{m-1}\lambda^k\alpha\bigr)
  =a_n+a_m,
 \end{split}
 \end{equation*}
 where the first inequality comes from condition \axhii{}, the second one from conditions \axli{} and \axhi{}, and the final from condition \axlii{}. Therefore $(a_n)$ is subadditive.
\end{proof}

By \cite{Wal82}*{Theorem 4.9}, if $(a_n)$ is a subadditive sequence of non-negative real numbers then the limit $\lim_n\frac{a_n}n=\inf_n\frac{a_n}n$ exists and is finite. Therefore, by Lemma \ref{lem:subatid} and Remark \ref{obs:h_infinito} we see that we can replace the ``$\limsup_n$'' in Definition \ref{def:entropy} by a ``$\lim_n$'' in the case of a lower map. We record this in the following statement.

\begin{lem}\label{lem:hfinite} 
 Let $\c$ be an entropy space, $\lambda\colon\c\to\c$ a lower map and $\alpha\in\c$. Then $h(\lambda,\alpha)=\lim_n\tfrac1nh(\alpha_0^{n-1})=\inf_n\tfrac1nh(\alpha_0^{n-1})$, and $h(\lambda,\alpha)$ is finite if $h$ is finite.
\end{lem}

We end this section with a related result we could not find in the literature.

\begin{lem}\label{lem:h_lambda_subaditiv} 
 If $\c$ is a commutative entropy space, $\lambda\colon\c\to\c$ is a  morphism, and $\alpha,\beta\in\c$, then $h(\lambda,\alpha\wedge\beta)\leq h(\lambda,\alpha)+h(\lambda,\beta)$.
\end{lem}

\begin{proof}
 Given $n\geq1$ we have $(\alpha\wedge\beta)_0^{n-1}\sim\alpha_0^{n-1}\wedge\beta_0^{n-1}$ by commutativity and condition \axmi{}. Then $h\bigl((\alpha\wedge\beta)_0^{n-1}\bigr)=h(\alpha_0^{n-1}\wedge\beta_0^{n-1})\leq h(\alpha_0^{n-1})+h(\beta_0^{n-1})$ by condition \axhii{}. Form this the claim follows dividing by $n$ and taking $\limsup_n$.
\end{proof}

\begin{rmk}\label{obs:h_hlambda}
 If $\c$ is a commutative entropy space and $\lambda\colon\c\to\c$ is a morphism, from Lemma \ref{lem:hlambda_monot} and Lemma \ref{lem:h_lambda_subaditiv} we deduce that the function $h^\lambda\colon\c\to\R^+$ given by $h^\lambda(\alpha)=h(\lambda,\alpha)$ for $\alpha\in\c$, is actually an entropy function for the cover space $\c$. We do not develop any application of this remark in this work but we point out that it gives a source of new entropy functions on a given entropy space. For example, as can be easily checked, any lower map $\mu$ for $(\c,h)$ such that $\mu\lambda\prec\lambda\mu$ (see condition \axliii{} in Definition \ref{def:connections}) is a lower map for $(\c,h^\lambda)$ too. Maybe it could be useful to study pairs $(\lambda,\mu)$ of commuting morphisms comparing the quantities $h^\lambda(\mu)$ and $h^\mu(\lambda)$ for the concrete entropy theories generalized here. 
\end{rmk}

\subsection{The logarithmic law}\label{subsec:loglaw}

\begin{pp}
 In this section we extend \cite{AKM65}*{Theorem 2} and its corollary to our abstract setting in Proposition \ref{prop:h_potencia} and Proposition \ref{prop:h_isomorfismo}. We also show, in Lemma \ref{lem:h_isomorfismo}, the ``bilateral formula'' for the entropy relative to a cover in the invertible case. 
\end{pp}

\begin{lem}\label{lem:h_potencia}
 Let $\c$ be an entropy space, $\lambda\colon\c\to\c$ a map, $\alpha\in\c$ and $m\geq1$. Then we have\quad1.\;\;$mh(\lambda,\alpha)\geq h(\lambda^m,\alpha)$.\quad If in addition  $\lambda$ is a lower map then we also have\quad2.\;\;$mh(\lambda,\alpha)\leq h(\lambda^m,\alpha_0^{m-1}[\lambda])$,\quad and\quad3.\;\;$h(\lambda,\lambda^m\alpha)\leq h(\lambda,\alpha)$.
\end{lem}

\begin{proof}
 Firstly, we estimate $\alpha_0^{nm-1}[\lambda]=\twedge_{k=0}^{nm-1}\lambda^k\alpha\prec\twedge_{k=0}^{n-1}\lambda^{mk}\alpha=\alpha_0^{n-1}[\lambda^m]$, where we applied Remark \ref{obs:cunatoria}. Therefore, by condition \axhi{} we have
 \[
  \textstyle
  h(\lambda,\alpha)
  \geq\limsup_n\tfrac1{nm}h(\alpha_0^{nm-1}[\lambda])
  \geq\limsup_n\tfrac1{nm}h(\alpha_0^{n-1}[\lambda^m])
  =\tfrac1mh(\lambda^m,\alpha),
 \]
 and we are done. For the second assertion let $\beta=\alpha_0^{m-1}[\lambda]$. Then, if $n\geq1$ we have
 \[
  \beta_0^{n-1}[\lambda^m]
  =\dwedge_{k=0}^{n-1}\lambda^{mk}\beta
  =\dwedge_{k=0}^{n-1}\lambda^{mk}\dwedge_{l=0}^{m-1}\lambda^l\alpha\\
  \prec\dwedge_{k=0}^{n-1}\dwedge_{l=0}^{m-1}\lambda^{mk+l}\alpha
  =\alpha_0^{nm-1}[\lambda],
 \]
 where we applied conditions \axli{} and \axcii{}. Hence, as $\lambda$ and $\lambda^m$ are lower maps (see Remark \ref{obs:potencias}) we can apply Lemma \ref{lem:hfinite} and condition \axhi{} to get 
 \[
  \textstyle
  h(\lambda^m,\beta)
  =\lim_n\tfrac1nh(\beta_0^{n-1}[\lambda^m])
  \geq\lim_n\tfrac{m}{nm}h(\alpha_0^{nm-1}[\lambda])
  =mh(\lambda,\alpha),
 \]
which ends the proof. For the last statement let $\beta=\lambda^m\alpha$ and $n\geq1$. Then we have $\beta_0^{n-1}=\twedge_{k=0}^{n-1}\lambda^{k+m}\alpha\succ\lambda^m\alpha_0^{n-1}$ by condition \axli{}. Hence, by conditions \axhi{} and \axlii{} we obtain $h(\beta_0^{n-1})\leq h(\lambda^m\alpha_0^{n-1})\leq h(\alpha_0^{n-1})$, and the result follows. 
\end{proof}

\begin{prop}\label{prop:h_potencia}  
If $\c$ is an entropy space and $\lambda\colon\c\to\c$ a map then for $m\geq1$ we have $h(\lambda^m)\leq mh(\lambda)$. If $\lambda$ is a lower map then $h(\lambda^m)=mh(\lambda)$ for $m\in\N$.
\end{prop}

\begin{proof} 
 For all $\alpha\in\c$, by Lemma \ref{lem:h_potencia}(1), we know that $mh(\lambda,\alpha)\geq h(\lambda^m,\alpha)$. Then taking $\sup_{\alpha\in\c}$ we get the first claim. For the second, firstly suppose $m\geq1$. By Lemma \ref{lem:h_potencia}(2) we have $mh(\lambda,\alpha)\leq h(\lambda^m,\alpha_0^{m-1}[\lambda])\leq h(\lambda^m)$ for all $\alpha\in\c$. Taking $\sup_{\alpha\in\c}$ we obtain $h(\lambda^m)\geq mh(\lambda)$ and hence $h(\lambda^m)= mh(\lambda)$. Finally, for $m=0$ we have $\lambda^0=\id_\c$. As $\id_\c=\id_\c^2$ by what we already proved $h(\id_\c)=2h(\id_\c)$, hence $h(\id_\c)=0$ or $h(\id_\c)=\infty$. In the first case we are done, and in the second what we need to prove is $\infty=0\cdot\infty$ which is true according to our convention.
\end{proof}

\begin{lem}\label{lem:h_isomorfismo}
If $\c$ is an entropy space and $\lambda\colon\c\to\c$ an isomorphism, then we have $h(\lambda,\alpha)=\textstyle\lim_n\tfrac1{2n+1}h(\alpha_{-n}^n)$ for all $\alpha\in\c$.
\end{lem}

\begin{proof} Given $\alpha\in\c$ and $n\geq1$, by condition \axmi{} we have $\alpha_{-n}^n\sim\lambda^{-n}\alpha_0^{2n}$. Therefore, by conditions \axhi{} and \axmii{} we obtain $h(\alpha_{-n}^n)=h(\lambda^{-n}\alpha_0^{2n})=h(\alpha_0^{2n})$. Then we deduce $\lim_n\tfrac1{2n+1}h(\alpha_{-n}^n)=\lim_n\tfrac1{2n+1}h(\alpha_0^{2n})=h(\lambda,\alpha)$ as claimed.
\end{proof}

\begin{prop}\label{prop:h_isomorfismo}
Let $\c$ be a commutative entropy space and $\lambda\colon\c\to\c$ an isomorphism, then $h(\lambda)=h(\lambda^{-1})$ and moreover  $h(\lambda^m)=|m|h(\lambda)$ for $m\in\Z$.
\end{prop}

\begin{proof}
 Given $\alpha\in\c$ and $n\geq1$ we have $\alpha_{-n}^n[\lambda]\sim\alpha_{-n}^n[\lambda^{-1}]$ by commutativity, and hence $h(\alpha_{-n}^n[\lambda])=h(\alpha_{-n}^n[\lambda^{-1}])$ by condition \axhi{}. Thus, by Lemma \ref{lem:h_isomorfismo} we obtain $h(\lambda,\alpha)=h(\lambda^{-1},\alpha)$ for all $\alpha\in\c$ and therefore $h(\lambda)=h(\lambda^{-1})$. To complete the proof it is enough to apply Proposition \ref{prop:h_potencia}.
\end{proof}

\subsection{Comparison of entropies}\label{subsec:comparison}

\begin{pp}
 In this section we introduce the \emph{connections} between maps in order to compare their entropies. The tools developed will be useful to deduce properties of the concrete entropy theories from properties of the abstract entropy in \S\ref{sec:ejemplos}. 
\end{pp}

\begin{df}\label{def:connections}
 For $i=1,2$ let $\c_i$ be an entropy space and $\lambda_i\colon\c_i\to\c_i$ a map. We say that $\mu$ is a \emph{map from $\lambda_1$ to $\lambda_2$}, and denote $\mu\colon\lambda_1\to\lambda_2$ or $\lambda_1\stackrel{\mu}{\to}\lambda_2$, iff $\mu\colon\c_1\to\c_2$ is a map. Let $\lambda_1\stackrel{\mu}{\to}\lambda_2$ a map. We call $\mu$ a \emph{lower connection} iff $\mu$ is a lower map (conditions \axli{} and \axlii{}) verifying
 \begin{enumerate}[label={\textsc{[l\footnotesize\arabic*}]}]
 \setcounter{enumi}{2}
  \item\label{axl3} $\mu\lambda_1\alpha\prec\lambda_2\mu\alpha$ for all $\alpha\in\c_1$.
 \end{enumerate}
 Analogously, if $\mu$ is an upper map (conditions \axui{} and \axuii{}) and
 \begin{enumerate}[label={\textsc{[u\footnotesize\arabic*}]}]
 \setcounter{enumi}{2}
  \item\label{axu3} $\mu\lambda_1\alpha\succ\lambda_2\mu\alpha$ for all $\alpha\in\c_1$,
 \end{enumerate}
 then $\mu$ is called \emph{upper connection}. We say that $\mu$ is a \emph{connection} iff $\mu$ is a homomorphism (conditions \axmi{} and \axmii) and
 \begin{enumerate}[label={\textsc{[m\footnotesize\arabic*}]}]
 \setcounter{enumi}{2}
  \item\label{axm3} $\mu\lambda_1\alpha\sim\lambda_2\mu\alpha$ for all $\alpha\in\c_1$.
 \end{enumerate}
 A connection which is an isomorphism is called a \emph{conjugation}. If there exists a conjugation $\lambda_1\to\lambda_2$ we say that $\lambda_1$ and $\lambda_2$ are \emph{conjugate}.
\end{df}

\begin{rmk}\label{obs:connection}
 For $i=1,2,3$ let $\c_i$ be an entropy space and $\lambda_i\colon\c_i\to\c_i$ a map. If $\lambda_1\stackrel{\mu_1}{\to}\lambda_2\stackrel{\mu_2}{\to}\lambda_3$ are lower connections, upper connections or connections, then their composition is again a lower connection, upper connection or connection $\lambda_1\stackrel{\mu_2\mu_1}{\longrightarrow}\lambda_3$. If $\mu\colon\c_1\to\c_2$ is a bi-monotone map (see Remark \ref{obs:bi-nomo}) we have that $\mu\colon\lambda_1\to\lambda_2$ is a lower/\vspace{0mm}upper connection iff $\mu ^{-1}\colon\lambda_2\to\lambda_1$ is an upper/\vspace{0mm}lower connection. In particular, a bi-monotone connection is a conjugation. Finally, note that if $\lambda_1\stackrel{\mu}{\to}\lambda_2$ is a lower connection, upper connection or a connection and $\lambda_2$ is monotone, then $\lambda_1^k\stackrel{\mu}{\to}\lambda_2^k$ is a lower connection, upper connection or a connection for all $k\in\N$.
\end{rmk}

\begin{lem}\label{lem:h_comparison}
 For $i=1,2$ let $\c_i$ be an entropy space, $\lambda_i\colon\c_i\to\c_i$ a map such that $\lambda_2$ is monotone and $\alpha\in\c_1$. Then the following statements hold.
 \begin{enumerate}[label={\arabic*.}]
  \item If $\lambda_1\stackrel{\mu}{\to}\lambda_2$ is an upper connection then $h(\lambda_1,\alpha)\leq h(\lambda_2,\mu\alpha)$.
  \item If $\lambda_1\stackrel{\mu}{\to}\lambda_2$ is a lower connection then $h(\lambda_1,\alpha)\geq h(\lambda_2,\mu\alpha)$.
  \item If $\lambda_1\stackrel{\mu}{\to}\lambda_2$ is a connection then $h(\lambda_1,\alpha)=h(\lambda_2,\mu\alpha)$.
 \end{enumerate}
\end{lem}

\begin{proof}
 For the first assertion consider $n\geq1$ and let $\beta=\mu\alpha$. Then we have 
 \[
  \mu\alpha_0^{n-1}[\lambda_1]
  =\mu\twedge_{k=0}^{n-1}\lambda_1^k\alpha
  \succ\twedge_{k=0}^{n-1}\mu\lambda_1^k\alpha
  \succ\twedge_{k=0}^{n-1}\lambda_2^k\mu\alpha
  =\beta_0^{n-1}[\lambda_2],
 \]
 where we applied condition \axui{} in the first inequality and the fact that $\lambda_1^k\stackrel{\mu}{\to}\lambda_2^k$ is an upper connection (see Remark \ref{obs:connection}) together with conditions \axuiii{} and \axcii{} in the second. Hence, using conditions \axuii{} and \axhi{} we estimate $h(\alpha_0^{n-1}[\lambda_1])\leq h(\mu\alpha_0^{n-1}[\lambda_1])\leq h(\beta_0^{n-1}[\lambda_2])$, and we conclude that $h(\lambda_1,\alpha)\leq h(\lambda_2,\beta)$ as desired.

 For the second statement, note that the proof given before works with the inequalities reversed, but this time arguing with conditions \axli{}, \axlii{} and \axliii{} instead of \axui{}, \axuii{} and \axuiii{}, showing that $h(\lambda_1,\alpha)\geq h(\lambda_2,\mu\alpha)$ for all $\alpha\in\c_1$. Finally, the last assertion follows directly from the first two because to be a connection is the same as to be a lower and upper connection simultaneously. 
\end{proof}

\begin{prop}\label{prop:h_comparison}
 For $i=1,2$ let $\c_i$ be an entropy space and $\lambda_i\colon\c_i\to\c_i$ a map such that $\lambda_2$ is monotone. Then the following statements hold.
 \begin{enumerate}[label={\arabic*.}]
  \item If $\lambda_1\stackrel{\mu}{\to}\lambda_2$ is an upper connection then $h(\lambda_1)\leq h(\lambda_2)$.
  \item If $\lambda_1\stackrel{\mu}{\to}\lambda_2$ is a cofinal lower connection then $h(\lambda_1)\geq h(\lambda_2)$.
  \item If $\lambda_1\stackrel{\mu}{\to}\lambda_2$ is a cofinal connection then $h(\lambda_1)=h(\lambda_2)$.
  \item If $\lambda_1$ and $\lambda_2$ are conjugate then $h(\lambda_1)=h(\lambda_2)$.
 \end{enumerate}
\end{prop}

\begin{proof}
 For the first assertion simply take $\sup_{\alpha\in\c_1}$ in the inequality of Lemma \ref{lem:h_comparison}(1). The second follows similarly form Lemma \ref{lem:h_comparison}(2) applying Lemma \ref{lem:cofinal} when taking $\sup_{\alpha\in\c_1}$. The last two claims are consequence of the first two. 
\end{proof}


\subsection{Products, direct limits and coproducts}\label{subsec:moreprop}

\begin{pp}
 In this section we consider products, direct limits and coproducts of entropy spaces and compute the entropy of induced maps.
\end{pp}

\begin{con}\label{con:producto}
 Let $I$ be a non-empty finite set, for $i\in I$ let $\c_i$ be an entropy space and let $(\alpha_i)$ and $(\beta_i)$ denote generic elements of the Cartesian product $\c=\tprod_{i\in I}\c_i$. We introduce an entropy space structure in $\c$ as follows: we consider the product pre-order, the binary operation and the entropy function defined as $(\alpha_i)\prec(\beta_i)$ iff $\alpha_i\prec\beta_i$ for all $i\in I$, $(\alpha_i)\wedge(\beta_i)=(\alpha_i\wedge\beta_i)$, and $h\colon\c\to\R^+$, $h\bigl((\alpha_i)\bigr)=\sum_{i\in I}h(\alpha_i)$, respectively. Then $\c=\tprod_{i\in I}\c_i$ is said to be the \emph{product entropy space} of the family $(\c_i)_{i\in I}$ and the lower morphisms $\c\to\c_i$ given by $(\alpha_i)\mapsto\alpha_i$ are called \emph{canonical maps}. 

 If $(\lambda_i\colon\c_i\to\c_i)_{i\in I}$ is family of maps then the map $\tprod_{i\in I}\lambda_i\colon\c\to\c$ given by $(\alpha_i)\mapsto(\lambda_i\alpha_i)$ is called the \emph{induced product map}.
\end{con}

If in Construction \ref{con:producto} the maps $\lambda_i$ are monotone, l-maps, l-morphisms, u-maps, u-morphisms, homomorphisms, morphisms or isomorphisms, then so is $\tprod_{i\in I}\lambda_i$.

\begin{prop}\label{prop:h_producto}
 Let $I$ be a non-empty finite set, for $i\in I$ let $\c_i$ be an entropy space and $\lambda_i\colon\c_i\to\c_i$ a map. Then $h(\tprod_{i\in I}\lambda_i)\leq\sum_{i\in I} h(\lambda_i)$. If in addition $\lambda_i$ is lower map for all $i\in I$ then $h(\tprod_{i\in I}\lambda_i)=\sum_{i\in I} h(\lambda_i)$.
\end{prop}

\begin{proof}
 Denote $\lambda=\tprod_{i\in I}\lambda_i$. Given $\alpha=(\alpha_i)\in\tprod_{i\in I}\c_i$ and $n\geq1$ it can be shown that $\alpha_0^{n-1}[\lambda]=(\alpha_i{}_0^{n-1}[\lambda_i])$, hence $h(\alpha_0^{n-1}[\lambda])=\sum_{i\in I}h(\alpha_i{}_0^{n-1}[\lambda_i])$. Dividing by $n$, taking $\limsup_n$ and then $\sup_{\alpha}$ we get the inequality. If the maps $\lambda_i$ are lower maps, by Lemma \ref{lem:hfinite} we can take $\lim_n$ instead of $\limsup_n$ getting an equality. 
\end{proof}

\begin{con}\label{con:limite}
 We briefly describe a construction of direct limits for entropy spaces. Given a non-empty directed set $(I,\leq)$, a family $(\c_i)_{i\in I}$ of entropy spaces and $\Phi=(\varphi_{ij}\colon\c_i\to\c_j)_{i\leq j}$ a \emph{coherent} family of homomorphisms, that is $\varphi_{jk}\varphi_{ij}=\varphi_{ik}$ if $i\leq j\leq k$, consider $\bigsqcup_i\c_i$, the disjoint union of the $\c_i$'s, and let $\alpha_i\in\c_i$ and $\beta_j\in\c_j$ denote generic elements of this set. Let $\c=\nicefrac{\bigsqcup_i\!\!\c_i}{\ssim}$ be the quotient space by the equivalence relation $\ssim$ defined as $\alpha_i\ssim\beta_j$ iff there exists $k\geq i,j$ such that $\varphi_{ik}\alpha_i=\varphi_{jk}\beta_j$, and let $[\alpha_i]$ denote the equivalence class of $\alpha_i$. Define a pre-order in $\c$ by $[\alpha_i]\prec[\beta_j]$ iff $\varphi_{ik}\alpha_i\prec\varphi_{jk}\beta_j$ for some $k\geq i,j$, a binary operation $[\alpha_i]\wedge[\beta_j]=[\varphi_{ik}\alpha_i\wedge\varphi_{jk}\beta_j]$ where $k\geq i,j$, and an entropy function $h([\alpha_i])=h(\alpha_i)$. Then $\c$ is an entropy space  denoted $\c=\lim_i\c_i$ and called \emph{direct limit entropy space} of the $\c_i$'s. The \emph{canonical maps} $\c_i\to\c$, $\alpha_i\mapsto [\alpha_i]$, are homomorphisms.

 If  $(\lambda_i\colon\c_i\to\c_i)_{i\in I}$ is a family of maps \emph{compatible} with $\Phi$, that is, $\varphi_{ij}\lambda_i=\lambda_j\varphi_{ij}$ if $i\leq j$, then the \emph{induced limit map} $\lim_i\lambda_i\colon\lim_i\c_i\to\lim_i\c_i$ is $[\alpha_i]\mapsto[\lambda_i\alpha_i]$.
\end{con}

 If in Construction \ref{con:limite} the maps $\lambda_i$ are monotone, l-maps, l-morphisms, u-maps, u-morphisms, homomorphisms, morphisms or isomorphisms, then so is $\lim_i\lambda_i$.

\begin{pp}
 In the proof of the next result we have a first application of the tools of \S\ref{subsec:comparison}.
\end{pp}

\begin{prop}\label{prop:h_directllimit}
 For a directed family of monotone maps between entropy spaces $(\lambda_i\colon\c_i\to\c_i)_{i\in I}$ compatible with a coherent family of homomorphisms $(\varphi_{ij}\colon\c_i\to\c_j)_{i\leq j}$ we have $h(\lim_i\lambda_i)=\sup_ih(\lambda_i)=\lim_ih(\lambda_i)$.  
\end{prop}

\begin{proof}
 Denote $\c=\lim_i\c_i$, $\lambda=\lim_i\lambda_i$ and for $i\in I$ let $\mu_i\colon\c_i\to\c$ be the canonical map. For $i\in I$, $\lambda_i\stackrel{\mu_i}{\to}\lambda$ is a connection, then by Lemma \ref{lem:h_comparison}(3) we know that that $h(\lambda_i,\alpha_i)=h(\lambda,\mu_i\alpha_i)$ for all $\alpha_i\in\c_i$. As $\c=\bigcup_{i\in I}\mu_i\c_i$ we have
 \begin{equation*}
 \begin{split}
  h(\lambda)
  &=\sup_{\alpha\in\c}h(\lambda,\alpha)
  =\sup_{i\in I}\sup_{\alpha_i\in\c_i}h(\lambda,\mu_i\alpha_i)
  =\sup_{i\in I}\sup_{\alpha_i\in\c_i}h(\lambda_i,\alpha_i)
  =\sup_{i\in I}h(\lambda_i).
 \end{split}
 \end{equation*}
 Finally, note that $\lambda_i\stackrel{\varphi_{ij}}{\longrightarrow}\lambda_j$ is a connection for all  $i\leq j$, then $h(\lambda_i)\leq h(\lambda_j)$ by Proposition \ref{prop:h_comparison}(1). Thus the net $\bigl(h(\lambda_i)\bigr)_{i\in I}$ is increasing and we are done.
 \end{proof}

\begin{con}\label{con:coproduct}
 Let $(\c_i)_{i\in I}$ be a family of unital entropy spaces (see Remark \ref{obs:h_unital_po}) with units denoted ${\uno}$. The \emph{coproduct entropy space} of the $\c_i$'s is the set
 \[
  \toplus_{i\in I}\c_i=\bigl\{(\alpha_i)_{i\in I}\in\tprod_{i\in I}\c_i:\alpha_i\neq{\uno}\text{ only for finitely many }i\text{'s}\bigr\},
 \]
 with the entropy space structure given by the pre-order, the binary operation, and the entropy function defined for elements $(\alpha_i)$ and $(\beta_i)$ as $(\alpha_i)\prec(\beta_i)$ iff $\alpha_i\prec\beta_i$ for all $i\in I$, $(\alpha_i)\wedge(\beta_i)=(\alpha_i\wedge\beta_i)$, and $h\bigl((\alpha_i)\bigr)=\sum_{i\in I}h(\alpha_i)$. The \emph{canonical maps} $\c_i\to\c$, $\alpha_i\mapsto(\alpha_{ij})_{j\in I}$, where $\alpha_{ij}=\alpha_i$ if $j=i$ or ${\uno}$ otherwise, are homomorphisms for all $i\in I$.

 If $(\lambda_i\colon\c_i\to\c_i)_{i\in I}$ is a family of \emph{unital maps}, that is $\lambda_i{\uno}={\uno}$ for all $i\in I$, the \emph{induced coproduct map} is $\toplus_{i\in I}\lambda_i\colon\toplus_{i\in I}\c_i\to\toplus_{i\in I}\c_i$, $(\alpha_i)\mapsto(\lambda_i\alpha_i)$.
\end{con}

If in Construction \ref{con:coproduct} the maps $\lambda_i$ are monotone, l-maps, l-morphisms, u-maps, u-morphisms, homomorphisms, morphisms or isomorphisms, then so is $\toplus_{i\in I}\lambda_i$.

\begin{cor}\label{cor:h_coproduct}
 For a family of unital entropy spaces $(\c_i)_{i\in I}$ and a family of monotone unital maps $(\lambda_i\colon\c_i\to\c_i)_{i\in I}$ we have $h(\toplus_{i\in I}\lambda_i)\leq\sum_{i\in I}h(\lambda_i)$ with equality if $\lambda_i$ is a lower map for all $i\in I$.
\end{cor}

\begin{proof}
 The coproduct $\toplus_{i\in I}\c_i$ is (isomorphic to) the direct limit $\lim_F\tprod_{i\in F}\c_i$, where $F$ runs over the directed set $J$ of finite subsets of $I$ ordered by inclusion, $\tprod_{i\in F}\c_i$ is the product entropy space as in Construction \ref{con:producto}, and the coherent homomorphisms are $\varphi_{EF}\colon\tprod_{i\in E}\c_i\to\tprod_{i\in F}\c_i$, $\varphi_{EF}\bigl((\alpha_i)_{i\in E}\bigr)=(\alpha_{Ej})_{j\in F}$, for finite subsets $E\subseteq F$ of $I$ and $\alpha_{Ej}=\alpha_j$ if $j\in E$ or $\alpha_{Ej}={\uno}$ otherwise. The monotone map $\toplus_{i\in I}\lambda_i$ is (conjugate to) the one induced in the direct limit by the compatible family of monotone maps $(\tprod_{i\in F}\lambda_i)_{F\in J}$, where $\tprod_{i\in F}\lambda_i$ is the product of the maps $\lambda_i$ for $i\in F$ as in Construction \ref{con:producto}. By Proposition \ref{prop:h_producto} we know that $h(\tprod_{i\in F}\lambda_i)\leq\sum_{i\in F}h(\lambda_i)$. Then applying Proposition \ref{prop:h_comparison}(4) and Proposition \ref{prop:h_directllimit} we get $h(\toplus_{i\in I}\lambda_i)
 =h(\lim_F\tprod_{i\in F}\lambda_i)
 =\lim_Fh(\tprod_{i\in F}\lambda_i)
 \leq\lim_F\sum_{i\in F}h(\lambda_i)
 =\sum_{i\in I}h(\lambda_i)$
 as desired. The case of lower maps follows similarly because Proposition \ref{prop:h_producto} gives an equality.
\end{proof}

To end this section we introduce a special type of products of entropy spaces.

\begin{con}\label{con:f-producto} 
 Let $f\colon\R^+\!\!\times\R^+\to\R^+$ be a \emph{monotone} and \emph{subadditive} map, that is, for $a_1,a_2,b_1,b_2\in\R^+$ the following hold:
 \enlargethispage{5mm}
 \begin{enumerate}[label={\arabic*.}]
  \item $f(a_1,a_2)\leq f(b_1,b_2)$ if $a_1\leq b_1$ and $a_2\leq b_2$,
  \item $f(a_1+b_1,a_2+b_2)\leq f(a_1,a_2)+f(b_1,b_2)$.
 \end{enumerate}
 If we have entropy spaces $(\c_i,h_i)$ for $i=1,2$, we define the \emph{$f$-product} $\c_1\times_f\c_2$ as the usual product as in Construction \ref{con:producto} but endowed with the entropy function given by $h_f(\alpha_1,\alpha_2)=f\bigl(h_1(\alpha_1),h_1(\alpha_2)\bigr)$ if $\alpha_i\in\c_i$ for $i=1,2$. Given maps $\lambda_i\colon\c_i\to\c_i$ for $i=1,2$, the \emph{induced product map} is $\lambda_1\times\lambda_2\colon\c_1\times_f\c_2\to\c_1\times_f\c_2$, $\lambda_1\times\lambda_2(\alpha_1,\alpha_2)=(\lambda_1\alpha_1,\lambda_2\alpha_2)$ if $\alpha_i\in\c_i  $ for $i=1,2$.
\end{con}

If in Construction \ref{con:f-producto} the maps $\lambda_i$ are monotone, l-maps, l-morphisms, u-maps, u-morphisms, homomorphisms, morphisms or isomorphisms, then so is $\lambda_1\times\lambda_2$.

\begin{lem}\label{lem:f-producto}
 Let $f\colon\R^+\!\!\times\R^+\to\R^+$ be the monotone and subadditive map given by $f(a,b)=\log(e^a+e^b)$, if $a,b\in\R^+$. For $i=1,2$ let $(\c_i,h_i)$ be an entropy space and $\lambda_i\colon\c_i\to\c_i$ a lower map. Then the induced product map on $\c_1\times_f\c_2$ as in Construction \ref{con:f-producto} verifies $h_f(\lambda_1\times\lambda_2)=\max\{h_1(\lambda_1),h_2(\lambda_2)\}$. 
\end{lem}

\begin{proof}
 Given $\alpha\in\c_1$ and $\beta\in\c_2$ define $a_n=h_1(\alpha_0^{n-1}[\lambda_1])$ and $b_n=h_2(\beta_0^{n-1}[\lambda_2])$ for $n\geq1$. As $\lambda_1\times\lambda_2$ is a lower map because $\lambda_1$ and $\lambda_2$ are lower maps, if $\gamma=(\alpha,\beta)\in\c_1\times_f\c_2$, by Lemma \ref{lem:hfinite} we have 
 \[\textstyle
  h_f(\lambda_1\times\lambda_2,\gamma)
  =\lim_n\tfrac1nh_f(\gamma_0^{n-1}[\lambda_1\times\lambda_2])
  =\lim_n\tfrac1nf(a_n,b_n).
 \]
 Therefore, as $\lim_n\frac{a_n}n=h_1(\lambda_1,\alpha)$ and $\lim_n\frac{b_n}n=h_2(\lambda_2,\beta)$, by \cite{AKM65}*{Lemma, p.\,312} we obtain $h_f(\lambda_1\times\lambda_2,\gamma)=\max\{h_1(\lambda_1,\alpha),h_2(\lambda_2,\beta)\}$. Then, taking $\sup_{\gamma}$ in the last equality the result follows.
\end{proof}

\section{Abstract expansiveness}\label{sec:expansiveness}

\subsection{Generators and expansiveness}\label{subsec:expansiv}

\begin{pp}
 In this section we extend the notion of expansiveness from the theory of topological dynamical systems to our context (see \S\ref{subsubsec:classic_exp} for more details). Here we will be  interested mainly in morphisms on meet entropy spaces.
\end{pp}

\begin{df}\label{def:exp_gen}
 Let $\c$ be a cover space and $\lambda\colon\c\to\c$ a map. A cover $\alpha\in\c$ is called \emph{positive generator} (or \emph{positive expansivity cover}) for $\lambda$ iff for all $\beta\in\c$ there exists $m\in\N$ such that $\alpha_0^m[\lambda]\prec\beta$. If a positive generator $\alpha$ for $\lambda$ exists we call $\lambda$ \emph{positively expansive}, or \emph{positively $\alpha$-expansive}. If $\lambda$ is bijective, a cover $\alpha\in\c$ is said to be a \emph{generator} (or \emph{expansivity cover}) for $\lambda$ iff for every $\beta\in\c$ there exists $m\in\N$ such that $\alpha_{-m}^m[\lambda]\prec\beta$. If a generator $\alpha$ for $\lambda$ exists we say that $\lambda$ is \emph{expansive}, or \emph{$\alpha$-expansive}.
\end{df}

In our motivating example, that is, a topological dynamical system $(X,T)$ where $T$ is a continuous map on a compact space $X$, we have that the associated map $\lambda_T\colon\c_X\to\c_X$ (see the first paragraph of \S\ref{subsec:maps}) is positively expansive according to Definition \ref{def:exp_gen} iff $T$ is \emph{positively refinement expansive} as in \cite{AAM16}*{Definition 3.17}, which in turn is equivalent to the usual \emph{positive expansivity} of $T$ on a metric space (see \cite{AoH94}*{p.\,40}) when $X$ has the Hausdorff separation property. If $T$ is a homeomorphism an analogous statement relating expansivity of $\lambda_T$, refinement expansivity of $T$ and the usual expansivity on metric spaces holds, as explained in \cite{Ach21}*{\S1} and \cite{AAM16}*{Theorem 3.13 and 2.7}. In \cite{Ach21}*{Definition 2.4} a notion of expansivity coherent with the previous ones is given at the intermediate level of lattices between expansivity in cover spaces and refinement expansivity.

In \cite{KeR69}*{Definition 2.4} a definition of \emph{topological generator} is introduced in the context of compact Hausdorff spaces, showing that the existence of a topological generator is equivalent to expansivity on metric spaces in \cite{KeR69}*{Theorem 3.2}. For a general compact space, that notion of generator is more restrictive than the one corresponding to Definition \ref{def:exp_gen} at the topological level, the latter being the notion of \emph{refinement expansivity cover} of \cite{AAM16}, which in turn corresponds to the \emph{expansivity covers} of \cite{Ach21} at the lattice level.

\begin{pp}
 In the following proposition, we present some basic properties that we have borrowed from the theory of expansive dynamical systems. To avoid introducing new terminology when discussing maps between cover spaces, we will use the same nomenclature as in \S\ref{sec:entropy}, assuming that the entropy functions are trivially constant and equal to zero. Therefore, we should disregard any conditions related to the entropy functions.
\end{pp}

\begin{prop}
 Let $\lambda\colon\c\to\c$ and $\lambda_i\colon\c_i\to\c_i$, $i=1,2$, be maps between cover spaces. Then the following statements hold.
 \begin{enumerate}[label={\arabic*.}]
  \item If $\lambda_1$ is positively expansive, $\lambda_1\stackrel{\mu}{\to}\lambda_2$ is a cofinal upper connection and $\lambda_2$ is monotone, then $\lambda_2$ is positively expansive.
  \item If $\lambda_1$, $\lambda_2$ are monotone and conjugate then $\lambda_1$ is positively expansive iff $\lambda_2$ is positively expansive.
  \item The coproduct $\lambda_1\oplus\lambda_2$ is positively expansive iff $\lambda_1$ and $\lambda_2$ are positively expansive.
  \item Let $n\in\N$, $n>0$. If $\lambda^n$ is positively expansive then $\lambda$ is positively expansive. If $\lambda$ is a lower map the converse statement also holds.
 \end{enumerate}
 Statements analogous to 2., 3. and 4. hold for the property of expansiveness for isomorphisms of commutative cover spaces (in 4. we can take $n\in\Z$)
\end{prop}

\begin{proof}
 For the first assertion, let $\alpha_1$ be a positive generator for $\lambda_1$. We will show that $\alpha_2=\mu\alpha_1$ is a positive generator for $\lambda_2$. Given a cover $\beta_2\in\c_2$, as $\mu$ is cofinal, there exists $\beta_1\in\c_1$ such that $\mu\beta_1\prec\beta_2$. We know that $\twedge_{k=0}^m\lambda_1^k\alpha_1\prec\beta_1$ for some $m\in\N$ because $\alpha_1$ is a positive generator. Then
 \[
  \twedge_{k=0}^m\lambda_2^k\alpha_2
  =\twedge_{k=0}^m\lambda_2^k\mu\alpha_1
  \prec\twedge_{k=0}^m\mu\lambda_1^k\alpha_1
  \prec\mu\twedge_{k=0}^m\lambda_1^k\alpha_1
  \prec\mu\beta_1
  \prec\beta_2,
 \]
 where in the first inequality we used that $\lambda_1^k\stackrel{\mu}{\to}\lambda_2^k$ is an upper connection (see Remark \ref{obs:connection}) together with conditions \axuiii{} and \axcii{}, in the second inequality we applied condition \axui{}, and in the third inequality we utilized monotonicity of $\mu$. This proves that $\alpha_2$ is a positive generator as claimed.
 
 The second statement follows form the first. For the third statement we simply indicate that it is easily checked that $(\alpha_1,\alpha_2)$ is a positive generator for $\lambda_1\oplus\lambda_2$ iff $\alpha_i$ is a positive generator for $\lambda_i$ for $i=1,2$. Finally, the fourth assertion can be shown as in \cite{Ach21}*{Proposition 2.7}, which corresponds to the invertible case. 
\end{proof}

Next we turn our attention to the relationship between expansivity and entropy.

\begin{lem}\label{lem:exp_gen} 
Let $\c$ be a meet entropy space, $\lambda\colon\c\to\c$ a morphism, $\alpha\in\c$ and $m\in\N$. Then $h(\lambda,\alpha_0^m)=h(\lambda,\alpha)$, and $h(\lambda,\alpha_{-m}^m)=h(\lambda,\alpha)$ if $\lambda$ is an isomorphism.
\end{lem}

\begin{proof}
Cancelling repeated factors (see Remark \ref{obs:icomut_idemp_meet}) we get $(\alpha_0^m)_0^{n-1}\sim\alpha_0^{m+n-1}$ for all $n\geq1$. Then, $h(\lambda,\alpha_0^m)=\limsup_n\frac1nh\bigl((\alpha_0^m)_0^{n-1}\bigr)=\limsup_n\frac1nh(\alpha_0^{m+n-1})=\limsup_n\frac{m+n}n\frac1{m+n}h(\alpha_0^{m+n-1})=h(\lambda,\alpha)$. For the last statement, note that we have $(\alpha_{-m}^m)_0^{n-1}\sim\lambda^{-m}\alpha_0^{2m+n-1}$, then $h\bigl((\alpha_{-m}^m)_0^{n-1}\bigr)=h(\lambda^{-m}\alpha_0^{2m+n-1 })=h(\alpha_0^{2m+n-1})$ by condition \axmii{}, and the result follows similarly.
\end{proof}

The following result, which generalizes \cite{KeR69}*{Theorem 2.6}, indicates that it would be interesting to discover new examples of expansivity in the existing entropy theories in order to compute entropy. In \cite{ArH20} the authors start a work in this direction in the category of commutative rings.

\begin{prop}\label{prop:exp_gen}
Let $\c$ be a meet entropy space. If $\lambda\colon\c\to\c$ is a positively expansive morphism with positive generator $\alpha\in\c$, then $h(\lambda)=h(\lambda,\alpha)$, and this quantity is finite if $h$ is finite and $\lambda$ is a lower morphism. If $\lambda$ is an  expansive isomorphism with generator $\alpha\in\c$ then $h(\lambda)=h(\lambda,\alpha)$, and is finite if $h$ is finite.
\end{prop}

\begin{proof} 
For a positive generator $\alpha\in\c$ the subset $\c'=\{\alpha_0^m[\lambda]:m\in\N\}\subseteq\c$ is cofinal. Then by Lemma \ref{lem:cofinal} and Lemma \ref{lem:exp_gen} we see that $h(\lambda)=h(\lambda,\alpha)$. The finiteness claim follows from Lemma \ref{lem:hfinite}. The statement about expansive isomorphisms is obtained similarly by Lemma \ref{lem:h_isomorfismo}.
\end{proof}

Next we extend Definition \ref{def:exp_gen} and Proposition \ref{prop:exp_gen}. This will be used in \S\ref{subsec:shifts}.

\begin{df}\label{def:exp_gen2}
Let $\c$ a cover space, $\lambda\colon\c\to\c$ a map and $\a\subseteq\c$. Then $\a$ is called \emph{positive generator system} for $\lambda$ iff for all $\beta\in\c$ there exist $\alpha\in\a$ and $m\in\N$ such that $\alpha_0^m[\lambda]\prec\beta$. In that case $\lambda$ is called \emph{positively $\a$-expansive}. If $\lambda$ is bijective, $\a$ is said to be a \emph{generator system} for $\lambda$ and $\lambda$ is called \emph{$\a$-expansive} iff for every $\beta\in\c$ there exist $\alpha\in\a$ and $m\in\N$ such that $\alpha_{-m}^m[\lambda]\prec\beta$.
\end{df}

A proof similar to that of Proposition \ref{prop:exp_gen} shows the following result.

\begin{cor}\label{cor:exp_gen2}
Let $\c$ be a meet entropy space, $\a\subseteq\c$ and $\lambda\colon\c\to\c$ a positively $\a$-expansive morphism/$\a$-expansive isomorphism, then $\displaystyle h(\lambda)=\sup_{\alpha\in\a}h(\lambda,\alpha)$.
\end{cor}

The next result improves \cite{AAM16}*{Lemma 3.19} and \cite{Ach21}*{Lemma 4.2}, where the dynamical system is supposed to be invertible, whereas here we assume a weaker hypothesis. The statement says that, under certain circumstances, positive expansivity implies that a condition much stronger than the definition holds. It is a primitive version of the main theorem of \cite{CoK06}, which we extend in Proposition \ref{prop:exp+finito}.

\begin{prop}\label{prop:exp+cofinal}
 Let $\c$ be a meet cover space and $\lambda\colon\c\to\c$ a positively expansive cofinal morphism. Then there exists $\beta\in\c$ with the following properties:\quad1.\;\;$\lambda^n\beta\succ\lambda^{n+1}\beta$ if $n\in\N$,\quad2.\;\;for all $\gamma\in\c$ there exists $n\in\N$ such that $\lambda^n\beta\prec\gamma$.
\end{prop}

\begin{proof}
 Let $\alpha\in\c$ be a positive generator for $\lambda$. As $\lambda$ is cofinal, there exists $\gamma\in\c$ such that $\lambda\gamma\prec\alpha$. Then, as $\alpha$ is a positive generator, we can find $m\geq1$ such that $\alpha\wedge\lambda\alpha\wedge\cdots\wedge\lambda^{m-1}\alpha\prec\gamma$. Let $\beta=\alpha\wedge\lambda\alpha\wedge\cdots\wedge\lambda^{m-1}\alpha$. Applying $\lambda$ to the last inequality we get $\lambda\beta\prec\lambda\gamma\prec\alpha$. Then, by Remark \ref{obs:icomut_idemp_meet}, we have
 \[
 \alpha\wedge\lambda\alpha\wedge\lambda^2\alpha\wedge\cdots\wedge\lambda^m\alpha
 =\alpha\wedge\lambda\beta
 \sim\lambda\beta
 =\lambda\alpha\wedge\lambda^2\alpha\wedge\cdots\wedge\lambda^m\alpha,
 \]
 and in general, applying $\lambda^k$ to this equivalence for $k\in\N$, we get
 \[
 \lambda^k\alpha\wedge\lambda^{k+1}\alpha\wedge\lambda^{k+2}\alpha\wedge\cdots\wedge\lambda^{k+m}\alpha
 \sim\lambda^{k+1}\alpha\wedge\lambda^{k+2}\alpha\wedge\cdots\wedge\lambda^{k+m}\alpha.
 \]
 Then we see that in a product of $m+1$ consecutive iterates of $\alpha$ we can cancel the first factor. Therefore, for every $n\in\N$, cancelling repeatedly a first factor, we have
 \[
 \alpha\wedge\lambda\alpha\wedge\lambda^2\alpha\wedge\cdots\wedge\lambda^{n+m-1}\alpha
 \sim\lambda^{n}\alpha\wedge\lambda^{n+1}\alpha\wedge\cdots\wedge\lambda^{n+m-1}\alpha
 =\lambda^n\beta.
 \]
 From this, it is clear that $\lambda^n\beta\succ\lambda^{n+1}\beta$ if $n\in\N$, and also that $\lambda^n\beta$ will refine any given cover for a sufficiently large $n\in\N$, because $\alpha$ is a positive generator.
\end{proof}

\begin{cor}
 If $\c$ is a meet entropy space and $\lambda\colon\c\to\c$ is a positively expansive cofinal lower morphism, then $h(\lambda)=0$.
\end{cor}

\begin{proof}
 Let $\beta\in\c$ as in Proposition \ref{prop:exp+cofinal}. As $\lambda^n\beta\succ\lambda^{n+1}\beta$ if $n\in\N$, we have $h(\lambda^n\beta)\leq h(\lambda^{n+1}\beta)$ if $n\in\N$, by \axhi{}. On the other hand, by \axlii{}, the reverse inequality also holds, and then we obtain $h(\lambda^n\beta)=h(\beta)$ for all $n\in\N$. From the property $\lambda^n\beta\succ\lambda^{n+1}\beta$ if $n\in\N$, by Remark \ref{obs:icomut_idemp_meet}  we also get $\beta_0^n[\lambda]\sim\lambda^n\beta$, for all $n\in\N$. Hence $h(\beta_0^n[\lambda])=h(\lambda^n\beta)=h(\beta)$ if $n\in\N$, and therefore
 \[\textstyle
  h(\lambda,\beta)
  =\limsup_n\tfrac1nh(\beta_0^{n-1}[\lambda])
  =\limsup_n\tfrac1nh(\beta)=0.
 \]
 Now, applying Lemma \ref{lem:h_potencia}(3) we deduce $h(\lambda,\lambda^n\beta)=0$ for all $n\in\N$. Finally, as by Proposition \ref{prop:exp+cofinal}(2) the set $\{\lambda^n\beta:n\in\N\}$ is cofinal, we see that $h(\lambda)=0$, by Lemma \ref{lem:cofinal}.
\end{proof}

\subsection{Shifts}\label{subsec:shifts}

\begin{pp}
  In this section we consider shift maps in the setting of entropy spaces and compute their entropies in Proposition \ref{prop:h_shift}.
\end{pp}

\begin{df}
 Let $\c$ be a unital cover space and $I=\N$ or $\Z$. Consider the coproduct cover space $\c^I=\bigoplus_{i\in I}\c$, and define the map $s^I\colon\c^I\to\c^I$ given by $s^I(\alpha_i)_{i\in I}=(\alpha_{i-1})_{i\in I}$, where we agree that $\alpha_{-1}={\uno}$ for the case $I=\N$. Either of these two maps is called \emph{forward shift}. Given a unital map $\lambda\colon\c\to\c$ we also consider a map $s_\lambda^I\colon\c^I\to\c^I$, which coincides with $s^I$ if $\lambda=\id_\c$, defined by $s_\lambda^I(\alpha_i)_{i\in I}=(\lambda\alpha_{i-1})_{i\in I}$, that is, $s_\lambda^I=s^I\lambda^I=\lambda^I s^I$, where $\lambda^I\colon\c^I\to\c^I$ is the coproduct map $\lambda^I=\bigoplus_{i\in I}\lambda$, $\lambda^I(\alpha_i)_{i\in I}=(\lambda\alpha_i)_{i\in I}$. Finally, given a cover $\alpha\in\c$, define $\xbar\alpha=(\alpha_{0i})_{i\in I}\in\c^I$ where $\alpha_{0i}=\alpha$ if $i=0$ or $\alpha_{0i}={\uno}$ otherwise.
\end{df}

\begin{lem}\label{lem:shift}
 Let $\c$ be a unital cover space and $\lambda\colon\c\to\c$ a unital cofinal monotone map. Then $s_\lambda^\N$ is positively $\a_\N$-expansive in the sense of Definition \ref{def:exp_gen2}, where $\a_\N=\{\xbar\alpha\in\c^\N:\alpha\in\c\}$. If $\lambda$ is a unital bi-monotone bijection then $s_\lambda^\Z$ is $\a_\Z$-expansive, where $\a_\Z=\{\xbar\alpha\in\c^\Z:\alpha\in\c\}$.
\end{lem}

\begin{proof}
 Given $(\alpha_i)_{i\in\N}\in\c^\N$, let $m\in\N$ such that $\alpha_i={\uno}$ if $i>m$. As $\lambda$ is monotone and cofinal then so is $\lambda^k$ for all $k\in\N$. Hence there exist $\beta_k\in\c$ such that $\lambda^k\beta_k\prec\alpha_k$ for all $k\in\N$. Let $\beta=\twedge_{k=0}^m\beta_k$, so that $\lambda^k\beta\prec\alpha_k$ for $0\leq k\leq m$. Then we have $\xbar\beta\in\a_\N$ and $\xbar\beta_0^m[s_\lambda^\N]=(\beta,\lambda\beta,\ldots,\lambda^m\beta,{\uno},{\uno},\ldots)\prec(\alpha_i)_{i\in\N}$, showing that $s_\lambda^\N$ is $\a_\N$-expansive. The statement about  $s_\lambda^\Z$ follows similarly.
\end{proof}

 For a meet entropy space $\c$ we denote $h(\c)=\sup_{\alpha\in\c}h(\alpha)=h(\id_\c)$.

\begin{prop}\label{prop:h_shift}
 Let $\c$ be a unital meet entropy space and $\lambda\colon\c\to\c$ a unital cofinal lower morphism. Then $h(s_\lambda^\N)\leq h(\c)$, with equality if $\lambda$ is also a homomorphism. If $\lambda$ is a unital isomorphism then $h(s_\lambda^\Z)=h(\c)$. In particular for the forward shifts we have $h(s^\N)=h(s^\Z)=h(\c)$.
\end{prop}

\begin{proof}
 Let $\alpha\in\c$ and $\xbar\alpha\in\a_\N\subseteq\c^\N$ as in Lemma \ref{lem:shift}. For $n\geq1$ we have
 \[
 h(\xbar\alpha_0^{n-1}[s_\lambda^\N])
 =h\bigl((\alpha,\lambda\alpha,\ldots,\lambda^{n-1}\alpha,{\uno},{\uno},\ldots)\bigr)
 =\textstyle\sum_{k=0}^{n-1}h(\lambda^k\alpha)
 \leq nh(\alpha),
 \]
 where the last inequality, given by condition \axlii{}, is an equality if $\lambda$ is a homomorphism, by condition \axmii{}. Dividing by $n$ and taking $\limsup_n$ we get $h(s_\lambda^\N,\xbar\alpha)\leq h(\alpha)$. As $\lambda$ is a morphism one can check that $s_\lambda^\N$ is a morphism too. Then, by Lemma \ref{lem:shift} and Corollary \ref{cor:exp_gen2} we have $h(s_\lambda^\N)=\sup_{\alpha\in\c}h(s_\lambda^\N,\xbar\alpha)\leq\sup_{\alpha\in\c}h(\alpha)$, with equality if $\lambda$ is a homomorphism. The assertion on $s_\lambda^\Z$ follows similarly.
\end{proof}

\section{Examples and applications}\label{sec:ejemplos}

In what follows, unless stated otherwise, we use the following notation: $X$ denotes a topological space, $T\colon X\to X$ is a continuous map, and $\prec$ and $\wedge$ denote the refinement relation and the meet operation, respectively, on families of sets, as described in \S\ref{subsec:notation}.

\subsection{Topological entropy}\label{subsec:h_top}

\begin{pp}
 Let us recall the definition of \emph{topological entropy} introduced in \cite{AKM65}.
\end{pp}

\begin{df}\label{def:htop}
 Let $\alpha$ be an open cover of $X$. We define 
 \begin{equation*}
 \begin{gathered}
  N(\alpha)=\min\{|\beta|:\beta\subseteq\alpha\text{ and }X\subseteq\tcup\beta\}
  ,\qquad
  H(\alpha)=\log N(\alpha),\\
  \textstyle
  h_\textup{top}(T,\alpha)=\limsup_n\tfrac1n H(\twedge_{k=0}^{n-1}T^{-k}\alpha)
  \quad\text{and}\quad
  h_\textup{top}(T)=\sup_{\alpha}h_\textup{top}(T,\alpha),
 \end{gathered}
 \end{equation*}
 where the $\sup_\alpha$ is taken over all open covers $\alpha$ of $X$. Then $h_\textup{top}(T)$ is the \emph{topological entropy of} $T$. We call $h_\textup{top}(T,\alpha)$ the topological entropy of $T$ \emph{of size $\alpha$}.
\end{df}

\begin{df}\label{def:htop_abstract}
 Let $\c_X$ be the set of all open covers of $X$, $h_X\colon\c_X\to\R^+$ given by $h_X(\alpha)=H(\alpha)$ if $\alpha\in\c_X$, and $\lambda_T\colon\c_X\to\c_X$ the map $\lambda_T\alpha=T^{-1}\alpha$ if $\alpha\in\c_X$.
\end{df}

In the next easy to prove statement, we record that the topological entropy is a particular instance of the abstract entropy developed in \S\ref{sec:entropy}.

\begin{prop}\label{prop:htop_abstract}
 The $4$-tuple $(\c_X,\prec,\wedge,h_X)$ is a meet entropy space, $\lambda_T$ is a lower morphism, and $h_X(\lambda_T)=h_\textup{top}(T)$. Moreover, $h_X$ is finite iff $X$ is compact.
\end{prop}

As a consequence of Proposition \ref{prop:htop_abstract} we can derive various of the know properties of the topological entropy from the development made in \S\ref{sec:entropy}. Some properties are obtained immediately. For example, a direct application of Lemma \ref{lem:subatid} shows that, for any open cover $\alpha$ of $X$, the sequence $n\mapsto H(\twedge_{k=0}^{n-1}T^{-k}\alpha)$ in the definition of $h_\textup{top}(T,\alpha)$ is subadditive (see the proof of \cite{AKM65}*{Property 8}). Other properties need some little work to be obtained. For example, if $m\in\N$, from Proposition \ref{prop:h_potencia} we have $h_X\big((\lambda_T)^m\big)=mh_X(\lambda_T)$. Then, noticing that $(\lambda_T)^m=\lambda_{T^m}$, we get $h_\textup{top}(T^m)=mh_\textup{top}(T)$, which is \cite{AKM65}*{Theorem 2}. Similarly, when $T$ is a homeomorphism, from Propositon \ref{prop:h_isomorfismo} we obtain $h_\textup{top}(T^m)=|m|h_\textup{top}(T)$ for $m\in\Z$, as in \cite{AKM65}*{Corollary, p.\,312}.

On the other hand, some properties demand more elaborate arguments to be derived. To deal with these situations we use the tools for comparing entropies of \S\ref{subsec:comparison}. As a first example we consider the following result and proof, where $S\colon Y\to Y$ denotes a continuous map on a topological space $Y$ and $T\times S\colon X\times Y\to X\times Y$ is the continuous map given by $(T\times S)(x,y)=(Tx,Sy)$ if $(x,y)\in X\times Y$.

\begin{prop}\label{prop:htop_producto} 
 If $X$ and $Y$ are compact then $h_\textup{top}(T\times S)\leq h_\textup{top}(T)+h_\textup{top}(S)$.
\end{prop}

\begin{proof}
 Let $\mu\colon\c_X\oplus\c_Y\to\c_{X\times Y}$ be the map given by $\mu(\alpha,\beta)=\alpha\otimes\beta$ if $\alpha\in\c_X$ and $\beta\in\c_Y$, where $\alpha\otimes\beta=\{U\times V:U\in\alpha,V\in\beta\}$. As can be easily checked, $\mu$ is a lower morphism, that is, $\mu$ is monotone, preserves the operation $\wedge$, and $h_{X\times Y}(\mu\gamma)=h_{X\times Y}(\alpha\otimes\beta)\leq h_X(\alpha)+h_Y(\beta)=h(\gamma)$ if $\gamma=(\alpha,\beta)\in\c_X\oplus\c_Y$, where $h$ denotes the entropy function of $\c_X\oplus\c_Y$ as in Construction \ref{con:coproduct}.

 On the other hand, we have $\mu\circ(\lambda_T\oplus\lambda_S)=\lambda_{T\times S}\circ\mu$, therefore in particular $\lambda_T\oplus\lambda_S\stackrel{\mu}{\to}\lambda_{T\times S}$ is a lower connection. Moreover, as $X$ and $Y$ are assumed to be compact, it can be shown that every open cover of $X\times Y$ is refined by some open cover of the form $\alpha\otimes\beta$, where $\alpha\in\c_X$ and $\beta\in\c_Y$ (see the proof of \cite{AKM65}*{Theorem 3}), that is, $\mu$ is a cofinal map. Therefore, by Proposition \ref{prop:h_comparison}(2), we have $h_{X\times Y}(\lambda_{T\times S})\leq h(\lambda_T\oplus\lambda_S)=h_X(\lambda_T)+h_Y(\lambda_S)$, where the last equality comes from Corollary \ref{cor:h_coproduct}. Thus, $h_\textup{top}(T\times S)\leq h_\textup{top}(T)+h_\textup{top}(S)$.
\end{proof}

\begin{rmk}
 Proposition \ref{prop:htop_producto} and its proof can be extended to the product of arbitrary many continuous maps.
\end{rmk}

\begin{rmk}
 Proposition \ref{prop:htop_producto} corresponds to \cite{AKM65}*{Theorem 3} where an equality is claimed. However, as noted by Goodwyn in \cite{Goo71}, the proof given in \cite{AKM65} only shows the inequality we stated. In \cite{Goo71}*{Theorem 2} the equality is proved assuming both $X$ and $Y$ are compact Hausdorff spaces. 
\end{rmk}

Another example, showing how to use the results in \S\ref{subsec:comparison} comparing entropies, is the next property, which is an alternative formulation of \cite{AKM65}*{Theorem 5}.

\begin{prop}\label{prop:htop_quotient}
 If $Y$ is a topological space, $\pi\colon X\to Y$ a continuous onto map, and $S\colon Y\to Y$ a continuous map such that $\pi T=S\pi$, then $h_\textup{top}(S)\leq h_\textup{top}(T)$.
\end{prop}

\begin{proof}
 Let $\mu\colon\c_Y\to\c_X$ the map given by $\mu\alpha=\pi^{-1}\alpha$ if $\alpha\in\c_Y$. As can be easily checked, $\mu$ is monotone, preserves the operation $\wedge$, $h_X(\mu\alpha)= h_Y(\alpha)$ if $\alpha\in\c_Y$, and $\mu\lambda_S=\lambda_T\mu$, that is, $\lambda_S\stackrel{\mu}{\to}\lambda_T$ is a connection. Then, by Proposition \ref{prop:h_comparison}(1), we conclude that $h_Y(\lambda_S)\leq h_X(\lambda_T)$, or equivalently, $h_\textup{top}(S)\leq h_\textup{top}(T)$. 
\end{proof}

For the next application, we recall from \cite{Har72}*{\S1} that a subspace $Y\subseteq X$ is called \emph{extension closed} iff for every open cover $\beta$ of $Y$ there exists an open cover $\alpha$ of $X$ such that $\beta=Y\wedge\alpha$. As noted in \cite{Har72}*{\S1}, closed subspaces are extension closed, and the converse holds if $X$ is a Hausdorff space. Now we are ready to state the following minor extension of \cite{AKM65}*{Theorem 4}.

\begin{prop}\label{prop:htop_union}
 If $X_1,X_2$ are extension closed subspaces of $X$ such that $X=X_1\cup X_2$ and $T(X_i)\subseteq X_i$ for $i=1,2$, then $h_\textup{top}(T)=\max\{h_\textup{top}(T|_{X_1}),h_\textup{top}(T|_{X_2})\}$.
\end{prop}

\begin{proof}
 Let $i\in\{1,2\}$. To simplify the nomenclature let us denote the restriction of $T$ to $X_i$ as $T_i=T|_{X_i}$. Following the notation of Definition \ref{def:htop_abstract}, let $\mu_i\colon\c_X\to\c_{X_i}$ be the map $\mu_i\alpha=\alpha\wedge X_i$ if $\alpha\in\c_X$. As can be easily checked, $\mu_i$ is a morphism, $\mu_i\lambda_T=\lambda_{T_i}\mu_i$ and $h_{X_i}(\mu_i\alpha)\leq h_X(\alpha)$ if $\alpha\in\c_X$. Hence, $\lambda_T\stackrel{\mu_i}{\to}\lambda_{T_i}$ is a lower connection. Moreover, as $X_i$ is extension closed, we see that $\mu_i$ is an onto map, and in particular a cofinal map. Then, by Proposition \ref{prop:h_comparison}(2), we conclude that $h_X(\lambda_T)\geq h_{X_i}(\lambda_{T_i})$ for $i=1,2$. Consequently, by Proposition \ref{prop:htop_abstract}, we obtain $h_\textup{top}(T)\geq\max\{h_\textup{top}(T|_{X_1}),h_\textup{top}(T|_{X_2})\}$.
 
 For the reverse inequality, consider the $f$-product $\c_{X_1}\times_f\c_{X_2}$ as in Lemma \ref{lem:f-producto}, and the map $\mu\colon\c_X\to\c_{X_1}\times_f\c_{X_2}$ given by $\mu\alpha=(\mu_1\alpha,\mu_2\alpha)$ if $\alpha\in\c_X$. This map is a morphism and $\mu\lambda_T=(\lambda_{T_1}\times\lambda_{T_2})\mu$. Given $\alpha\in\c_X$ it can be easily shown that $N(\alpha\wedge X_1)+N(\alpha\wedge X_2)\geq N(\alpha)$, where the quantities $N$ are as in Definition \ref{def:htop}. Taking logarithms, we obtain $h_f(\mu\alpha)\geq h_X(\alpha)$ if $\alpha\in\c_X$, and therefore $\lambda_T\stackrel{\mu}{\to}\lambda_{T_1}\times\lambda_{T_2}$ is an upper connection. Thus, by Proposition \ref{prop:h_comparison}(1) and  Lemma \ref{lem:f-producto}, we have $h_X(\lambda_T)\leq h_f(\lambda_{T_1}\times\lambda_{T_2})=\max\{h_{X_1}(\lambda_{T_1}), h_{X_2}(\lambda_{T_2})\}$. Finally, an application of Proposition \ref{prop:htop_abstract} ends the proof. 
\end{proof}

We end this section obtaining part of \cite{Goo71}*{Theorem 1}, a result involving the topological entropy of inverse limits of dynamical systems. Following \cite{Goo71}*{\S3}, given a directed set $(I,\leq)$, a family of continuous maps $(T_i\colon X_i\to X_i)_{i\in I}$, where $X_i$ is a topological space if $i\in I$, and a collection of continuous maps $(S_{ij}\colon X_i\to X_j)_{i\geq j}$ such that $S_{ij}\circ T_i=T_j\circ S_{ij}$ for all $i,j\in I$, $i\geq j$, and $S_{jk}S_{ij}=S_{ik}$ for all $i,j,k\in I$, $i\geq j\geq k$, we define the \emph{inverse limit} of this system as the continuous map $T\colon X\to X$, where $X=\{(x_i)_{i\in I}\in\tprod_{i\in I}X_i:S_{ij}x_i=x_j\text{ for all }i,j\in I,i\geq j\}$ with the relative product topology, and $T\big((x_i)_{i\in I}\big)=(T_ix_i)_{i\in I}$ if $(x_i)_{i\in I}\in X$.

\begin{prop}\label{prop:htop_invlim}
 In the context of the preceding paragraph, if the spaces $X_i$ are compact and the maps $S_{ij}$ are onto, then $h_\textup{top}(T)=\sup_ih_\textup{top}(T_i)=\lim_ih_\textup{top}(T_i)$.
\end{prop}

\begin{proof}
 For each $i\in I$ define $(\c_i,h_i)=(\c_{X_i},h_{X_i})$, the entropy space associated to the space $X_i$, and $\lambda_i\colon\c_i\to\c_i$, $\lambda_i=\lambda_{T_i}$, the lower morphism associated to $T_i$, as in Definition \ref{def:htop_abstract}. For $i\in I$ let $\pi_i\colon X\to X_i$ be the restriction to $X$ of the canonical projection of $\tprod_{i\in I}X_i$ on $X_i$, and $\mu_i\colon\c_i\to\c_X$ the map $\mu_i\alpha_i=\pi_i^{-1}\alpha_i$ if $\alpha_i\in\c_i$. As the maps $S_{ij}$ are assumed to be onto maps, the maps $\pi_i$ are onto maps too, and then the maps $\mu_i$ are homomorphisms. For $i,j\in I$, $i\leq j$, consider the map $\varphi_{ij}\colon\c_i\to\c_j$ given by $\varphi_{ij}\alpha_i=S_{ji}^{-1}\alpha_i$ if $\alpha_i\in\c_i$.
 
 As can be easily checked, $\Phi=(\varphi_{ij})_{i\leq j}$ is a coherent family of homomorphisms and the family $(\lambda_i)_{i\in I}$ is compatible with $\Phi$ as in Construction \ref{con:limite}. Then we can take the direct limits $\c=\lim_i\c_i$ and $\lambda=\lim_i\lambda_i$. If $i,j, k\in I$ are such that $i,j\leq k$, and $\alpha_i\in\c_i$, $\beta_j\in\c_j$ verifies $\varphi_{ik}\alpha_i=\varphi_{jk}\beta_j$, it can be shown that $\mu_i\alpha_i=\mu_j\beta_j$. Then, the map $\bigsqcup_i\c_i\to\c_X$, given by $\alpha_i\mapsto\mu_i\alpha_i$ if $\alpha_i\in\c_i$, is compatible with the equivalence relation $\ssim$ of Construction \ref{con:limite}, and therefore induces a map $\mu\colon\c\to\c_X$. The map $\mu$ is a homomorphism, and is a cofinal map by \cite{Goo71}*{Lemma 3}. Moreover, as $\mu_i\lambda_i=\lambda_T\mu_i$ for all $i\in I$ we obtain $\mu\lambda=\lambda_T\mu$, so that $\lambda\stackrel{\mu}{\to}\lambda_T$ is a cofinal connection. Hence, by Propositions \ref{prop:h_comparison}(3) and \ref{prop:h_directllimit}, we get $h_X(\lambda_T)=h(\lambda)=\sup_ih_i(\lambda_i)=\lim_ih_i(\lambda_i)$, and the result follows from Proposition \ref{prop:htop_abstract}.
\end{proof}

\subsection{Topological mean dimension}\label{subsec:mdim}

\begin{pp}
 Let us recall the definition of \emph{mean dimension} introduced in \cite{LiW00}*{Definition 2.6}.
\end{pp}

\begin{df}\label{def:mdim}
 Let $\alpha$ be an open cover of $X$. We define 
 \begin{equation*}
 \begin{gathered}
  \ord(\alpha)=\max\{|\beta|:\beta\subseteq\alpha\text{ and }\tcap\beta\neq\varnothing\}-1,\\
  D(\alpha)=\min\{\ord(\beta):\beta\prec\alpha\text{ and }\tcup\beta=X\},\\
  \textstyle
  \mdim(T,\alpha)=\limsup_n\tfrac1n D(\twedge_{k=0}^{n-1}T^{-k}\alpha)
  \quad\text{and}\quad
  \mdim(T)=\sup_{\alpha}\mdim(T,\alpha),
 \end{gathered}
 \end{equation*}
 where the $\sup_\alpha$ is taken over all open covers $\alpha$ of $X$. Then $\mdim(T)$ is the \emph{mean dimension of} $T$ as in \cite{LiW00}. We call $\mdim(T,\alpha)$ the mean dimension of $T$ \emph{of size $\alpha$}.
\end{df}

Consider $\c_X$ and $\lambda_T$ as in Definition \ref{def:htop_abstract} but this time define $h_X$ as follows.

\begin{df}\label{def:mdim_abstract}
Let $h_X\colon\c_X\to\R^+$ be the map $h_X(\alpha)=D(\alpha)$ if $\alpha\in\c_X$.
\end{df}

The next statement says that mean dimension is a particular instance of the abstract entropy developed in \S\ref{sec:entropy}.

\begin{prop}\label{prop:mdim_abstract}
 The $4$-tuple $(\c_X,\prec,\wedge,h_X)$ is a meet entropy space, $\lambda_T$ is a lower morphism, and $h_X(\lambda_T)=\mdim(T)$. Moreover, $h_X$ is finite if $X$ is compact.
\end{prop}

\begin{proof}
 The only difficult fact to check is the property $D(\alpha\wedge\beta)\leq D(\alpha)+D(\beta)$ if $\alpha,\beta\in\c_X$. This is proved in \cite{LiW00}*{Corollary 2.5}.
\end{proof}

As in the case of topological entropy, form Proposition \ref{prop:mdim_abstract} and the results in \S\ref{sec:entropy}, we can derive a series of properties of the mean dimension. For example, subadditivity of the sequence $n\mapsto D(\twedge_{k=0}^{n-1}T^{-k}\alpha)$, the logarithmic law $\mdim(T^n)=n\mdim(T)$ (see \cite{LiW00}*{Proposition 2.7}, where the property is stated only for $n\geq1$), and the bound for the mean dimension of products, that is \cite{LiW00}*{Proposition 2.8}, can be shown as in \S\ref{subsec:h_top}.

\begin{pp}
 Next we introduce a result, not stated in \cite{LiW00}, containing the ``two-sided'' formula for the mean dimension of size $\alpha$ of a homeomorphism, and the generalization of the logarithmic law for integer exponents.
\end{pp}

\begin{prop}\label{prop:mdim_loglaw}
 If $T$ is a homeomorphism the following statements hold.
 \begin{enumerate}[label={\arabic*.}]
  \item $\mdim(T,\alpha)=\lim_n\tfrac1{2n+1}D(\twedge_{|k|\leq n}T^k\alpha)$, if $\alpha\in\c_X$.
  \item $\mdim(T^n)=|n|\mdim(T)$ if $n\in\Z$.
 \end{enumerate}
\end{prop}

\begin{proof}
 As $T$ is a homeomorphism, then $\lambda_T\colon\c_X\to\c_X$ is an isomorphism. Therefore, form Lemma \ref{lem:h_isomorfismo} we get the first assertion, and the second one from Proposition \ref{prop:h_isomorfismo}.
\end{proof}

\subsection{Topological expansiveness}\label{subsec:topexp}

\subsubsection{Classical expansive systems}\label{subsubsec:classic_exp}
 The notion of \emph{expansive homeomorphism} was introduced by Utz \cite{Utz50} in 1950 under the name \emph{unstable homeomorphism}. Since then, the subject was vastly studied by many authors, developing what nowadays is known as the theory of expansive dynamical systems.

 If $X$ is a metrizable space, $d$ is a compatible metric, and $T$ is a homeomorphism, then $T$ is said to be \emph{expansive} iff there exists a constant $\varepsilon_d>0$, called \emph{expansivity constant}, such that if $x,y\in X$ and $d(f^nx,f^ny)<\varepsilon_d$ for all $n\in\Z$ then $x=y$. When $X$ is compact this property does not depend on the choice of $d$, so that expansiveness is a topological property. In the literature, several purely topological equivalent definitions of expansiveness was introduced. Among them, we consider the characterization given in \cite{KeR69}. 
 
 If $X$ is a topological space and $T$ is a homeomorphism, we say that an open cover $\alpha$ of $X$ is a \emph{strong generator} iff for every bi-sequence $(U_n)_{n\in\Z}$ of elements of $\alpha$ the set\footnote{Here $\xbar U$ denotes the closure of $U$.} $\bigcap_{n\in\Z}T^n\xbar{U}_{\!n}$ contains at most one point. This is \cite{KeR69}*{Definition 2.4}, where a strong generator is called simply generator. We added the word ``strong'' to distinguish this concept from the one corresponding to Definition \ref{def:exp_gen}, which we discuss in the next paragraph. Keynes and Robertson proved in \cite{KeR69}*{Theorem 3.2} that if $X$ is compact, the existence of a strong generator is equivalent to expansivity (in particular, the space being metrizable).
 
 If $X$ is a topological space and $T$ is a homeomorphism, we say that an open cover $\alpha$ of $X$ is a \emph{generator} iff for every open cover $\beta$ of $X$ there exists $n\in\N$ such that $\twedge_{|k|\leq n}T^k\alpha\prec\beta$. In \cite{AAM16} generators are called \emph{r-expansive covers}, and if a generator exists $T$ is called \emph{refinement expansive} (see \cite{AAM16}*{Definition 3.1, Definition 3.6 and Theorem 3.9}). If $X$ is compact, every strong generator is a generator, as shown in \cite{KeR69}*{Lemma 2.5}. The converse of the last statement is not true, because there are examples of refinement expansive homeomorphisms on non-Hausdorff (hence non metrizable) compact spaces, as the non-Hausdorff shifts of \cite{AAM16}*{\S4.1}. However, if $X$ is a compact Hausdorff space then refinement expansivity and expansivity are equivalent.
  
 It is easy to check that an open cover $\alpha$ is a generator in the sense of the preceding paragraph iff $\alpha$ is a generator, in the sense of Definition \ref{def:exp_gen2}, for the map $\lambda_T\colon\c_X\to\c_X$ of Definition \ref{def:htop_abstract}. Then we see that expansivity and, in general, refinement expansivity for a homeomorphism $T$ of a compact space $X$, are particular cases of the notion of expansivity of Definition \ref{def:exp_gen}.
 
 In \cite{KeR69}*{Theorem 2.6} it is shown that if $X$ is compact and $\alpha$ is a strong generator then $h_\textup{top}(T)=h_\textup{top}(T,\alpha)$. In fact, the same proof works if $\alpha$ is a generator, so the equality still holds in this case. This result is a special case of Proposition \ref{prop:exp_gen}.
 
 The discussion above also applies, with the necessary adaptations, to positive expansiveness. Let $T\colon X\to X$ be a continuous map on a topological $X$. Then $T$ is called \emph{positively expansive} iff there exist a compatible metric $d$ and $\varepsilon_d>0$ such that if $x,y\in X$ and $d(f^nx,f^ny)<\varepsilon_d$ for all $n\in\N$ then $x=y$. An open cover $\alpha$ of $X$ is called a \emph{positive generator} iff for every open cover $\beta$ of $X$ there exists $n\in\N$ such that $\twedge_{k=0}^nT^{-k}\alpha\prec\beta$. If a positive generator exists, $T$ is called \emph{positive refinement expansive}. Positive expansivity implies positive refinement expansivity on compact spaces, and on compact Hausdorff spaces both notions coincide. Finally, we have $h_\textup{top}(T)=h_\textup{top}(T,\alpha)$ if $\alpha$ is a positive generator. Again, these things are particular cases of Definition \ref{def:exp_gen} and Proposition \ref{prop:exp_gen}.

\subsubsection{Positively expansive embeddings}\label{subsubsec:exp+finito}

 In the next application, we will use the concept of extension closed subspace already discussed before Proposition \ref{prop:htop_union}. Recall that $X$ is a $T_1$ space iff $\{x\}$ is closed for every $x\in X$. We say that $T$ is an \emph{embedding} iff $T$ is a homeomorphism onto its image $T(X)$.

\begin{prop}\label{prop:exp+finito}
 If $X$ is a compact $T_1$ space, $T$ is a positively refinement expansive embedding and $T(X)$ is extension closed, then $X$ is finite.
\end{prop}

\begin{proof}
 Given an open cover $\alpha$ of $X$ we have that $T\alpha$ is an open cover of $T(X)$. Then, as $T(X)$ is extension closed, there exists an open cover $\beta$ of $X$ such that $T\alpha=T(X)\wedge\beta$. We also have $T^{-1}\beta=\alpha$. Then we showed that the morphism $\lambda_T\colon\c_X\to\c_X$ is an onto map, and in particular a cofinal map. Moreover, as $T$ is positively refinement expansive, $\lambda_T$ is positively expansive. Therefore, by Proposition \ref{prop:exp+cofinal} there exists an open cover $\beta$ of $X$ such that the set $\{T^{-n}\beta:n\in\N\}$ is cofinal. Then the proof continues as in \cite{AAM16}*{Theorem 3.20}.
\end{proof}

\begin{rmk}
 Proposition \ref{prop:exp+finito} improves \cite{AAM16}*{Theorem 3.20}, where the positive refinement expansive map is supposed to be a homeomorphism, and in turn extends the main theorem of \cite{CoK06}: \emph{Every compact metric space that supports a continuous, one-to-one, positively expansive map is finite}. This is because clearly a one-to-one continuous map on a compact metric space is an embedding with closed image.
\end{rmk}

\subsubsection{Topological generator systems}\label{subsubsec:top_gen_syst}

Applying Definition \ref{def:exp_gen2} to the map $\lambda_T$ of Definition \ref{def:htop_abstract} we obtain the following notions of \emph{topological generator systems} for the continuous map $T$. 

\begin{df}
 Let $\a$ be a collection of open covers of $X$. Then $\a$ is called a \emph{positive generator system} for $T$ iff for every open cover $\beta$ of $X$ there exist $\alpha\in\a$ and $n\in\N$ such that $\twedge_{k=0}^nT^{-k}\alpha\prec\beta$. If $T$ is a homeomorphism, then $\a$ is called a \emph{generator system} for $T$ iff for every open cover $\beta$ of $X$ there exist $\alpha\in\a$ and $n\in\N$ such that $\twedge_{|k|\leq n}T^{-k}\alpha\prec\beta$.
\end{df}

This definition extends the concepts of generator and positive generator discussed in \S\ref{subsubsec:classic_exp}, which correspond to the case in which the collection $\a$ has only one element. The next result generalizes to this context the well known fact, also mentioned in \S\ref{subsubsec:classic_exp}, that $h_\textup{top}(T)=h_\textup{top}(T,\alpha)$ if $\alpha$ is a generator or a positive generator for $T$. Note that this generalization take into account not only the topological entropy but also the mean dimension.

\begin{prop}\label{prop:exp_gen2_top}
 If $\a$ is a generator system or a positive generator system for $T$ then\quad$h_\textup{top}(T)=\sup_{\alpha\in\a}h_\textup{top}(T,\alpha)$\quad and\quad$\mdim(T)=\sup_{\alpha\in\a}\mdim(T,\alpha)$.
\end{prop}

\begin{proof}
 The statement is a consequence of Corollary \ref{cor:exp_gen2} applied to the map $\lambda_T$ acting on the entropy space $(\c_X,h_X)$, where $h_X$ is either the entropy function associated to the topological entropy theory as Definition \ref{def:htop_abstract}, or the one associated to the mean dimension theory as in Definition \ref{def:mdim_abstract}.
\end{proof}

\begin{ex}[Shifts]\label{ex:shift_top}
 Let $I=\N$ or $\Z$, $\Sigma_I=X^I$ with the product topology, and $\sigma_I\colon\Sigma_I\to\Sigma_I$ the \emph{shift} map, $\sigma_I\bigl((x_i)_{i\in I}\bigr)=(x_{i+1})_{i\in I}$ for $(x_i)_{i\in I}\in\Sigma_I$. Let $\pi_I\colon\Sigma_I\to X$ be the projection onto the zeroth coordinate, $\pi_I\bigl((x_i)_{i\in I}\bigr)=x_0$ if $(x_i)_{i\in I}\in\Sigma_I$, and define $\a_I=\{\pi_I^{-1}\alpha:\alpha\in\c_X\}$. As can be easily checked, if $X$ is compact then $\a_\Z$ is a generator system for $\sigma_\Z$ and $\a_\N$ is a positive generator system for $\sigma_\N$.
\end{ex}

In the next result, where $X$ is assumed to be compact and $\dim(X)$ denotes his topological (covering) dimension, the statement about $\sigma_\Z$ is \cite{LiW00}*{Proposition 3.1}.

\begin{prop}\label{prop:mdim_shift}
 We have\quad $\mdim(\sigma_\Z)\leq\dim(X)$\quad  and\quad  $\mdim(\sigma_\N)\leq\dim(X)$.
\end{prop}

\begin{proof}
 First, let us consider the case of $\sigma_\N$. By Proposition \ref{prop:exp_gen2_top} and Example \ref{ex:shift_top} we have $\mdim(\sigma_\N)=\sup_{\beta\in\a_\N}\mdim(\sigma_\N,\beta)$. Then it is enough to show that $\mdim(\sigma_\N,\beta)\leq\dim(X)$ for every $\beta\in\a_\N$. To do that we introduce the following notation. Given open covers $\alpha_0,\ldots,\alpha_n$ of topological spaces $X_0,\ldots,X_n$ denote $\alpha_0\otimes\cdots\otimes\alpha_n=\{U_0\times\cdots\times U_n:U_i\in\alpha_i\text{ for }i=0,\ldots,n\}$, which is an open cover of the product space $X_0\times\cdots\times X_n$. If $D$ is the function introduced in Definition \ref{def:mdim}, we have $D(\alpha_0\otimes\cdots\otimes\alpha_n)\leq D(\alpha_0)+\cdots+D(\alpha_n)$. Indeed, for example, for the case $n=2$ we have $D(\alpha_0\otimes\alpha_1)=D(\alpha_0\otimes\{X_1\}\wedge\{X_0\}\otimes\alpha_1)\leq D(\alpha_0\otimes\{X_1\})+D(\{X_0\}\otimes\alpha_1)=D(\alpha_0)+D(\alpha_1)$, where the inequality comes from the subadditivity of $D$.
 
 Given $\beta\in\a_\N$ let $\alpha\in\c_X$ be such that $\beta=\alpha\otimes\{\tprod_{k=1}^\infty X\}$. If $n\geq1$ we have 
 \[
 \twedge_{k=0}^{n-1}\sigma_\N^{-k}\beta
 =\underbrace{\alpha\otimes\cdots\otimes\alpha}_{\text{$n$ times}}\otimes\,\{\tprod_{k=n}^\infty X\},
 \]
 so that $D(\twedge_{k=0}^{n-1}\sigma_\N^{-k}\beta)\leq nD(\alpha)+D(\{\tprod_{k=n}^\infty X\})=nD(\alpha)$. Therefore,
 \[
 \mdim(\sigma_\N,\beta)
 =\textstyle\limsup_n\tfrac1nD(\twedge_{k=0}^{n-1}\sigma_\N^{-k}\beta)\leq D(\alpha)\leq\dim(X),
 \]
 where the last inequality holds by the definition $\dim(X)=\sup_{\alpha\in\c_X}D(\alpha)$ of the covering dimension of $X$. The statement about $\sigma_\Z$ follows similarly but using the formula of Proposition \ref{prop:mdim_loglaw}(1) to compute $\mdim(\sigma_\Z,\beta)$ for $\beta\in\a_\Z$.
\end{proof}

\begin{rmk}
 Proposition \ref{prop:mdim_shift} could be derived also from Proposition \ref{prop:h_shift}. 
\end{rmk}

\subsection{Forward topological entropy}\label{subsec:fhtop}

\begin{pp}
  In this section, we consider an open onto map $T\colon X\to X$ on a topological space $X$. We introduce the concept of \emph{forward topological entropy} for such maps and establish some basic properties using the general theory presented in \S\ref{sec:entropy}. We conclude by providing a few examples. Further developments on this topic, such as obtaining a definition \emph{à la Bowen}, will be discussed in a subsequent paper.
\end{pp}

\begin{df}\label{def:fhtop}
 Consider the quantity $H$ of Definition \ref{def:htop} and let $\alpha$ be an open cover of $X$. We define
 \begin{equation*}
 \begin{gathered}
  \textstyle
  h^*_\textup{top}(T,\alpha)=\limsup_n\tfrac1n H(\twedge_{k=0}^{n-1}T^k\alpha)
  \quad\text{and}\quad
  h^*_\textup{top}(T)=\sup_{\alpha}h^*_\textup{top}(T,\alpha),
 \end{gathered}
 \end{equation*}
 where the $\sup_\alpha$ is taken over all open covers $\alpha$ of $X$. Then $h^*_\textup{top}(T)$ is called the \emph{forward topological entropy of} $T$.
\end{df}

Consider $\c_X$ and $h_X$ as in Definition \ref{def:htop_abstract} but this time define $\lambda_T$ as follows.

\begin{df}\label{def:fhtop_abstract}
 Let $\lambda_T\colon\c_X\to\c_X$ be the map $\lambda_T\alpha=T\alpha$ if $\alpha\in\c_X$.
\end{df}

The following statement is easy to prove and records that the forward topological entropy is a specific instance of the abstract entropy developed in \S\ref{sec:entropy}.

\begin{prop}\label{prop:fhtop_abstract}
 The $4$-tuple $(\c_X,\prec,\wedge,h_X)$ is a meet entropy space, $\lambda_T$ is a lower map, and $h_X(\lambda_T)=h^*_\textup{top}(T)$. Moreover, $h_X$ is finite iff $X$ is compact.
\end{prop}

\begin{rmk}\label{obs:fhtop=htop}
 If $T$ is a homeomorphism, then we have $h^*_\textup{top}(T)=h_\textup{top}(T)$. This is because the definition of $h^*_\textup{top}(T)$ in this case is identical to the definition of $h_\textup{top}(T^{-1})$, and it is known that $h_\textup{top}(T^{-1})=h_\textup{top}(T)$ (see \cite{AKM65}*{Corollary, p.\,312}).
\end{rmk}

\begin{prop}\label{prop:fhtop}
 Let $\alpha,\beta\in\c_X$. Then the following statements hold. 
 \begin{enumerate}[label={\arabic*.}]
  \item If $\alpha\prec\beta$ then $h^*_\textup{top}(T,\alpha)\geq h^*_\textup{top}(T,\beta)$.
  \item  If $\c'\subseteq\c_X$ is cofinal, $h^*_\textup{top}(T)=\sup_{\alpha\in\c'}h^*_\textup{top}(T,\alpha)=\lim_{\alpha\in\c'}h^*_\textup{top}(T,\alpha)$.
  \item The sequence $n\mapsto H(\twedge_{k=0}^{n-1}T^k\alpha)$ is subadditive.
  \item We have $h^*_\textup{top}(T,\alpha)=\lim_n\tfrac1nH(\twedge_{k=0}^{n-1}T^k\alpha)=\inf_n\tfrac1nH(\twedge_{k=0}^{n-1}T^k\alpha)$.
  \item If $X$ is compact then $h^*_\textup{top}(T,\alpha)$ is finite. 
 \end{enumerate}
\end{prop}

\begin{proof}
 As $h^*_\textup{top}(T,\alpha)=h_X(\lambda_X,\alpha)$ and similarly for $\beta$, by Proposition \ref{prop:fhtop_abstract} and Lemma \ref{lem:hlambda_monot} it follows (1). Analogously, (2) is a consequence of Lemma \ref{lem:cofinal}, (3) follows from Lemma \ref{lem:subatid}, and from Lemma \ref{lem:hfinite} we obtain (4) and (5).
\end{proof}

For an open onto map $S\colon Y\to Y$ on a topological space $Y$, note that the product map $T\times S\colon X\times Y\to X\times Y$, $(T\times S)(x,y)=(Tx,Sy)$ if $(x,y)\in X\times Y$, is an open onto map as well. Recall the concept of extension closed subspace discussed before Proposition \ref{prop:htop_union}.

\begin{thm}\label{teo:fhtop}
 Let $Y$ be a topological space, $S\colon Y\to Y$ an open onto map, $\pi\colon X\to Y$ a continuous open onto map, and $X_1$, $X_2$ extension closed subspaces of $X$ such that $X=X_1\cup X_2$ and $T^{-1}(X_i)\subseteq X_i$ for $i=1,2$. Then the following statements hold. 
 \begin{enumerate}[label={\arabic*.}]
  \item If $m\in\N$ then $h^*_\textup{top}(T^m)=mh^*_\textup{top}(T)$.
  \item If $X$ and $Y$ are compact then $h^*_\textup{top}(T\times S)\leq h^*_\textup{top}(T)+h^*_\textup{top}(S)$.
  \item If $\pi T=S\pi$ then $h^*_\textup{top}(S)\leq h^*_\textup{top}(T)$.
  \item  We have $h^*_\textup{top}(T)=\max\{h^*_\textup{top}(T|_{X_1}),h^*_\textup{top}(T|_{X_2})\}$.
 \end{enumerate}
\end{thm}

\begin{proof}
 For $m\in\N$ it is easily checked that $\lambda_{T^m}=(\lambda_T)^m$. Then by Proposition \ref{prop:fhtop_abstract} and Proposition \ref{prop:h_potencia} we obtain
 \[
  h^*_\textup{top}(T^m)
  =h_X(\lambda_{T^m})
  =h_X\bigl((\lambda_T)^m\bigr)
  =mh_X(\lambda_T)
  =mh^*_\textup{top}(T),
 \]
 from which the first assertion follows.
 
 To show (2), consider the cofinal lower morphism $\mu\colon\c_X\oplus\c_Y\to\c_{X\times Y}$ defined in the proof of Proposition \ref{prop:htop_producto}. It is easily checked that $\mu\circ(\lambda_T\oplus\lambda_S)=\lambda_{T\times S}\circ\mu$. (Note that here $\lambda_T$, $\lambda_S$, $\lambda_{T\times S}$ have a different meaning than in Proposition \ref{prop:htop_producto}.) Then, in particular $\lambda_T\oplus\lambda_S\stackrel{\mu}{\to}\lambda_{T\times S}$ is a cofinal lower connection. Hence, we have 
 \[
  h_\textup{top}^*(T\times S)
   =h_{X\times Y}(\lambda_{T\times S})
   \leq h(\lambda_T\oplus\lambda_S)
   =h_X(\lambda_T)+h_Y(\lambda_S)
   =h_\textup{top}^*(T)+h_\textup{top}^*(S),
 \]
 where the inequality comes from Proposition \ref{prop:h_comparison}(2), the middle equality is given by Corollary \ref{cor:h_coproduct}, and $h$ is the entropy function of $\c_X\oplus\c_Y$. This proves (2).
 
 For the proof of (3), define $\mu\colon\c_X\to\c_Y$ as $\mu\alpha=\pi\alpha$ if $\alpha\in\c_X$, which is a well defined map because $\pi$ is open and onto. It can be checked that $\mu$ is a lower map, that is, $\mu$ is monotone, $\mu(\alpha\wedge\beta)\prec\mu\alpha\wedge\mu\beta$ and $h_Y(\mu\alpha)\leq h_X(\alpha)$, if $\alpha,\beta\in\c_X$. As $\pi T=S\pi$ we obtain that $\mu\lambda_T=\lambda_S\mu$, then, in particular, $\lambda_T\stackrel{\mu}{\to}\lambda_S$ is a lower connection. Finally, since $\pi$ is also assumed to be continuous, we have that $\mu$ is onto and, in particular, a cofinal map. Therefore, from Proposition \ref{prop:h_comparison}(2) we obtain $h_\textup{top}^*(T)=h_X(\lambda_T)\geq h_Y(\lambda_S)=h_\textup{top}^*(S)$, and we are done.
 
 The proof of (4) is identical, word for word, to the proof of Proposition \ref{prop:htop_union}, but instead of using Proposition \ref{prop:htop_abstract}, we argue with Proposition \ref{prop:fhtop_abstract}. Note that in this context, the meanings of $\lambda_T$, $\lambda_{T_1}$, and $\lambda_{T_2}$ are different from those in the cited proof. Here, the equality $\mu_i\lambda_T=\lambda_{T_i}\mu_i$, for $i=1,2$, is guaranteed by the assumption that $T^{-1}(X_i)\subseteq X_i$ for $i=1,2$. From this, we also have $\mu\lambda_T=(\lambda_{T_1}\times\lambda_{T_2})\mu$.
\end{proof}

\begin{ex}\label{ex:fhtop_shift}
 For a compact Hausdorff space $X$, such that $|X|\geq2$, we have
 \[
  h_\textup{top}(\sigma_{\Z})=h_\textup{top}^*(\sigma_{\Z})=\log|X|
  \qquad\text{and}\qquad
  h_\textup{top}(\sigma_{\N})\neq h_\textup{top}^*(\sigma_{\N})=0, 
 \]
 where for $I=\N$ or $\Z$ the map $\sigma_I$ (or $\sigma_{I,X}$) is the shift map as in Example \ref{ex:shift_top}. 
 
 Indeed, it is well known that $h_\textup{top}(\sigma_{I,X})=\log|X|$. Then, for $I=\Z$, by Remark \ref{obs:fhtop=htop} we get the first equality. On the other hand, for $I=\N$, given any open subset $U\subseteq X^\N$ there exists $n\in\N$ such that $\sigma_{\N}^n(U)=X^\N$. Hence, for each open cover $\alpha$ of $X^\N$ there is some $m\in\N$ such that $\sigma_{\N}^k\alpha=\{X^\N\}$ for all $k\geq m$. Therefore, $\twedge_{k=0}^{n-1}\sigma_\N^k\alpha$ is constant for $n\geq m$, which implies that $\tfrac1nH(\twedge_{k=0}^{n-1}\sigma_\N^k\alpha)\to0$. Then, $h_\textup{top}^*(\sigma_{\N},\alpha)=0$ for all open covers $\alpha$, and consequently $h_\textup{top}^*(\sigma_{\N})=0$ as claimed.
\end{ex}

\begin{ex}\label{ex:fhtop_z^n}
 For $n\in\Z$, let $z^n$ be the $n$-winding map on the circle $\s^1\subseteq\C$. Then
 \[
 h_\textup{top}^*(z^n)=0\;\;\text{for}\;\;n\neq0
 \qquad\text{and}\qquad
 h_\textup{top}(z^n)=\log|n|.
 \]
 
 The second assertion is well known. For the first one, note that it is obvious if $|n|=1$, while for $|n|\geq 2$ an argument similar to that of Example \ref{ex:fhtop_shift} works, because any open set of the circle is mapped to the whole $\s^1$ by some power of $z^n$.
\end{ex}

\begin{ex}\label{ex:fhtop_distinguish}
 In this example we show continuous open onto maps $T$ and $S$ such that $h_\textup{top}^*(T)\neq h_\textup{top}^*(S)$ and $h_\textup{top}(T)=h_\textup{top}(S)$. This means that sometimes $h_\textup{top}^*$ can ``see'' where $h_\textup{top}$ can not.
  
 Let $X,Y,Z$ be compact Hausdorff spaces such that $|X|=\infty$ and $|Y|\neq|Z|\geq2$. In the notation of Example \ref{ex:fhtop_shift}, define $T=\sigma_{\N,X}\sqcup\sigma_{\Z,Y}$, $S=\sigma_{\N,X}\sqcup\sigma_{\Z,Z}$, where ``$\sqcup$'' represents the \emph{disjoint union}. Then, by Theorem \ref{teo:fhtop}(4) and Example \ref{ex:fhtop_shift} we have $h_\textup{top}^*(T)=\max\{h_\textup{top}^*(\sigma_{\N,X}),h_\textup{top}^*(\sigma_{\Z,Y})\}=\max\{0,\log|Y|\}=\log|Y|$, and similarly $h_\textup{top}^*(S)=\log|Z|$. Therefore $h_\textup{top}^*(T)\neq h_\textup{top}^*(S)$. On the other hand, by Proposition \ref{prop:htop_union}, we obtain $h_\textup{top}(T)=h_\textup{top}(S)=\infty$.
\end{ex}

\subsection{Algebraic entropies and expansiveness}\label{subsec:algebraic}

\begin{pp}
 In this section, we will provide an overview of some algebraic entropy theories that fit into the framework of our abstract entropy. Additionally, we will propose a few new definitions of entropy and expansiveness within the algebraic context, with the aim of suggesting potential lines for further research. Although we do not have any particular applications or examples in mind, we hope that sharing these ideas will be helpful to others interested in exploring this direction.
\end{pp}

\subsubsection{Algebraic entropies}\label{subsubsec:h_alg}

The notion of \emph{algebraic entropy} for endomorphisms of abelian groups is introduced in \cite{AKM65}*{\S5} and first studied in \cite{Weiss74}. Given an abelian group $G$ let $\c_G$ be the set of all finite subgroups of $G$, and consider on $\c_G$ the relation $\prec$ and the binary operation $\wedge$ defined for $E,F\in\c_G$ as: $E\prec F$ iff $F\subseteq E$, and $E\wedge F=E+F$ (the subgroup generated by $E\cup F$), respectively. Also define $h_G\colon\c_G\to\R^+$ by $h(F)=\log|F|$ if $F\in\c_G$. Then $(\c_G,\prec,\wedge,h_G)$ is a meet entropy space (where actually $\prec$ is a partial order). If in addition $\varphi\colon G\to G$ is an endomorphism let $\lambda_\varphi\colon\c_G\to\c_G$ be the map given by $\lambda_\varphi F=\varphi F$ if $F\in\c_G$. Then $\lambda_\varphi$ is a lower morphism (see \cite{Weiss74}*{Proposition 1.1}) and the abstract entropy $h_G(\lambda_\varphi)$ is precisely the algebraic entropy $h(\varphi)$ defined in \cite{Weiss74}*{Definition 1.1}. We call this entropy \emph{Weiss' algebraic entropy}.

Later, the above definition was modified in \cite{Pet79} for automorphisms and in \cite{DGB16} for endomorphisms, where $\c_G$ was replaced by the set of all finite \emph{subsets} of $G$. The resulting notion of entropy is referred to as \emph{Peters' algebraic entropy}.

To mention a last algebraic example, we consider the \emph{adjoint algebraic entropy} for group endomorphisms introduced in \cite{DGS10}. This entropy is obtained by taking $\c_G$ as the set of all subgroups of finite index of an abelian group $G$, $\prec$ and $\wedge$ the relation and the operation on $\c_G$ given by $E\prec F$ iff $E\subset F$ and $E\wedge F=E\cap F$, respectively, and $\lambda_\varphi$ as the map taking the elements of $\c_G$ to its preimage under the endomorphism $\varphi\colon G\to G$. Again, $(\c_G,\prec,\wedge,h_G)$ is a meet entropy space and $\lambda_\varphi$ is a lower morphism. The abstract entropy $h_G(\lambda_\varphi)$ obtained with the above choices is the adjoint algebraic entropy $\textrm{ent}^\star(\varphi)$ defined in \cite{DGS10}*{p.\,5}.

We refer the reader to \cite{DGB19} for more examples of purely algebraic nature.

\subsubsection{Backward algebraic entropy}\label{subsubsec:backh_alg}

In all of the examples discussed in \S\ref{subsubsec:h_alg}, the function $\lambda$ preserves the operation $\wedge$, making them particular instances of the semigroup entropy of \cite{DGB19}. Next, we propose a modification to Weiss' algebraic entropy where $\lambda$ no longer preserves $\wedge$, while still remaining a lower map.

Consider an \emph{injective} group endomorphism $\varphi\colon G\to G$, and let $\c_G$, $\prec$, $\wedge$, and $h_G$ be defined as in the Weiss' algebraic entropy case. However, in this case, we define $\lambda_\varphi\colon\c_G\to\c_G$ by $\lambda_\varphi F=\varphi^{-1}F$ if $F\in\c_G$. While $(\c_G,\prec,\wedge,h_G)$ remains a meet entropy space as before, $\lambda_\varphi$ is now merely a lower map. The resulting entropy, which could be called \emph{backward Weiss' algebraic entropy}, will share properties with the abstract entropy because it fits within our context.

This ``procedure of reversing the direction of entropies'' is analogous to what we have done with the topological entropy obtaining a forward version of it in \S\ref{subsec:fhtop}, and could potentially be applied to other entropies as well, such as Peters' or adjoint entropies.

This "procedure of reversing the direction of entropies" is similar to what we did with the topological entropy by obtaining a forward version of it in \S\ref{subsec:fhtop}. It may also be possible to apply this procedure to other entropies, such as Peters' or adjoint entropies.

\subsubsection{Algebraic expansivity}\label{subsubsec:exp_alg}

Given a group endomorphism $\varphi\colon G\to G$, if we apply Definition \ref{def:exp_gen} to the map $\lambda_G$ associated to the Weiss' algebraic entropy as in \S\ref{subsubsec:h_alg}, we obtain a definition of \emph{positively expansive group endomorphism}. The endomorphism $\varphi$ is positively expansive iff there exists a finite subgroup $F$ of $G$, a \emph{positive generator}, such that for every finite subgruop $E$ of $G$ there exists $m\in\N$ such that $E\subseteq\sum_{k=0}^m\varphi^kF$. If $G$ is a torsion group this amounts to $\sum_{k=0}^{+\infty}\varphi^kF=G$.

While the Bernoulli shifts, $\sigma\colon H^{\oplus\N}\to H^{\oplus\N}$, $\sigma(h_0,h_1,\ldots)=(0,h_0,h_1,\ldots)$, where $H$ is a finite abelian group, are positively expansive, Anna Giordano Bruno sketched a proof to me that if $\sum_{k=0}^{+\infty}\varphi^kF=G$, then it implies that $\varphi$ is essentially a shift of that type. Then the way we defined positively expansive endomorphisms may not be interesting. However, it could be possible to obtain interesting notions of ``algebraic expansivity" by applying Definition \ref{def:exp_gen} in the context of other entropies.

\vspace{10mm}

\noindent Mauricio Achigar\\
{\tt machigar@unorte.edu.uy}\\
{\sc Departamento de Matemática y Estadística del Litoral}\\
{\sc Centro Universitario Regional Litoral Norte}\\
{\sc Universidad de la República}\\
25 de Agosto 281, Salto (50000), Uruguay

\end{document}